\documentclass[onefignum,onetabnum]{article}

\usepackage{lipsum}
\usepackage{amsfonts}
\usepackage{amsthm}
\usepackage{graphicx}
\usepackage{epstopdf}
\usepackage{caption}
\usepackage{xcolor}
\usepackage{hyperref}
\usepackage{subcaption}
\usepackage{algorithm, algorithmic}
\usepackage[a4paper, total={6in, 8in}]{geometry}
\ifpdf
  \DeclareGraphicsExtensions{.eps,.pdf,.png,.jpg}
\else
  \DeclareGraphicsExtensions{.eps}
\fi
\usepackage{cite} %% Pierre : pour avoir references dans l'ordre 
\usepackage{amsmath} %% Pierre : pour align

\title{A Directional Equispaced interpolation-based\\Fast Multipole Method for oscillatory kernels}

\author{Igor Chollet\thanks{Institut des Sciences du Calcul et des Donn\'ees (ISCD), Sorbonne Universit\'e, INRIA Alpines, F-75005 (igor.chollet@inria.fr).}
\and Xavier Claeys\thanks{Laboratoire Jacques-Louis Lions, Sorbonne Universit\'e, Inria équipe ALPINES, F-75005
  (claeys@ann.jussieu.fr).}
\and Pierre Fortin \thanks{Sorbonne Universit\'e, CNRS, LIP6, F-75005
Paris, France; Univ. Lille, CNRS, Centrale Lille, UMR 9189 CRIStAL, F-59000 Lille, France (pierre.fortin@univ-lille.fr)}
\and Laura Grigori \thanks{Laboratoire Jacques-Louis Lions, Sorbonne Universit\'e, Inria équipe ALPINES, F-75005 (laura.grigori@inria.fr)}}

\usepackage{amsopn}

\makeatletter
\newcommand*{\addFileDependency}[1]{% argument=file name and extension
  \typeout{(#1)}% latexmk will find this if $recorder=0 (however, in that case, it will ignore #1 if it is a .aux or .pdf file etc and it exists! if it doesn't exist, it will appear in the list of dependents regardless)
  \@addtofilelist{#1}% if you want it to appear in \listfiles, not really necessary and latexmk doesn't use this
  \IfFileExists{#1}{}{\typeout{No file #1.}}% latexmk will find this message if #1 doesn't exist (yet)
}
\makeatother

\newtheorem{theorem}{Theorem}
\newtheorem{lemma}{Lemma}
\newtheorem{definition}{Definition}
\newcommand{\nc}{\newcommand}
\nc{\bb}{\mathbb}
\nc{\cc}{\mathcal}
\nc{\kk}{\mathfrak}
\nc{\bo}{\mathbf}
\nc{\scr}{\mathcal}

\nc{\inpar}[1]{\ensuremath{\left( #1 \right)}}
\nc{\meif}[3]{\ensuremath{\begin{aligned}\begin{cases}#1 &\textit{if }#2\\0&\textit{otherwise}\end{cases}\end{aligned}}}

\nc{\meiff}[5]{\ensuremath{\begin{aligned}\begin{cases}#1 &\textit{if }#2\\#3&\textit{if }#4\\#5&\textit{otherwise}\end{cases}\end{aligned}}}

\nc{\fc}[1]{\bb{C}[#1]}
\nc{\scal}[1]{\langle #1\rangle}
\nc{\mymap}[3]{#1\hspace{0.05cm}:\hspace{0.05cm}#2\rightarrow #3}

\nc{\cblue}[1]{\color{blue}{#1}\color{black}{}}
\nc{\cred}[1]{\color{red}{#1}\color{black}{}}
\nc{\cmblue}[1]{\color{MidnightBlue}{#1}\color{black}{}}
\nc{\cgrey}[1]{\color{gray}{#1}\color{black}{}}
\nc{\cmagenta}[1]{\color{magenta}{#1}\color{black}{}}
\nc{\corange}[1]{\color{orange}{#1}\color{black}{}}
\nc{\cviolet}[1]{\color{violet}{#1}\color{black}{}}

\begin{document}

\maketitle

\begin{abstract}
Fast Multipole Methods (FMMs) based on the oscillatory Helmholtz kernel can reduce the cost of solving N-body problems arising from Boundary Integral Equations (BIEs) in acoustic or electromagnetics. However, their cost strongly increases in the high-frequency regime. 
  This paper introduces a new directional FMM for oscillatory kernels (\textit{defmm} - directional equispaced interpolation-based fmm), whose precomputation and application are FFT-accelerated due to polynomial interpolations on equispaced grids. We demonstrate the consistency of our FFT approach, and show how symmetries can be exploited in the Fourier domain. We also describe the algorithmic design of \textit{defmm}, well-suited for the BIE non-uniform particle distributions, and present performance optimizations on one CPU core. Finally, we exhibit important performance gains on all test cases for \textit{defmm} over a state-of-the-art FMM library for oscillatory kernels.
\end{abstract}

\section{Introduction}
\label{s:intro}
Considering two point clouds $X,Y\subset \bb{R}^d$ with cardinal $N\in \bb{N}^*$, $d\in \bb{N}^*$ and $q\in \bb{C}[Y]$ (named \textit{charges}), where $\bb{C}[Y]$ denotes the set of application from $Y$ to $\bb{C}$, we are interested in the fast computation of $p\in \bb{C}[X]$ (referred to as \textit{potentials}) such that
\begin{equation}
\label{fmmsum}
    p(\bo{x}) := \sum_{\bo{y}\in Y}G(\bo{x},\bo{y}) q(\bo{y}),\hspace{0.5cm}\forall\hspace{0.05cm}\bo{x}\in X,
\end{equation}
where $G\hspace{0.05cm}:\hspace{0.05cm}\bb{R}^d\times \bb{R}^d\rightarrow \bb{C}$. Such an \textit{N-body problem} appears in the numerical solving of Boundary Integral Equations (BIEs). We are especially concerned by the oscillatory \textit{Helmholtz kernel} involved in BIEs applied to acoustic or electromagnetics
\begin{equation*}
    G(\bo{x},\bo{y}) := \frac{e^{i\kappa |\bo{x}-\bo{y}|}}{4\pi |\bo{x}-\bo{y}|},
\end{equation*}
where $\kappa \in \bb{R}^+$ is named the \textit{wavenumber}, $i$ denotes the complex number and $|\cdot|$ refers to the Euclidian norm. Computing $p$ using hierarchical methods can achieve an $\cc{O}(N\hspace{0.02cm}\log\hspace{0.02cm}N)$ complexity on large surface \textit{particle distributions} (i.e. point clouds), that are the distributions generated when discretizing BIEs. However, the cost of the hierarchical methods can still be a bottleneck in the \textit{high-frequency regime},
i.e. when $\kappa D \gg 1$, $D$ denoting the side length of the smallest box encompassing $X$ and $Y$.
\subsection{Related work}
The direct evaluation of the N-body problem \eqref{fmmsum} has a complexity $\cc{O}(N^2)$. This complexity can be reduced using a hierarchical method, such as the Fast Multipole Method (FMM) \cite{rokhlingreengard87,wideband06}. 
Thanks to a hierarchical decomposition, such as a $2^d$-tree\footnote{Binary tree (d=1), quadtree (d=2) or octree (d=3).} representation of the (cubical) computational domain $B$, and to a multipole acceptance criterion, the computation of \eqref{fmmsum} is indeed divided into two parts: a near field one which is computed directly and a far field one which is 
approximated through multipole and local expansions. Nodes of such a $2^d$-tree are named \textit{cells} and 
correspond to cubical subdomains of $B$ whose radii depend on the tree level (see Figure \ref{fig:octree}). 

\begin{figure}
    \centering
    \includegraphics[width=0.8\linewidth]{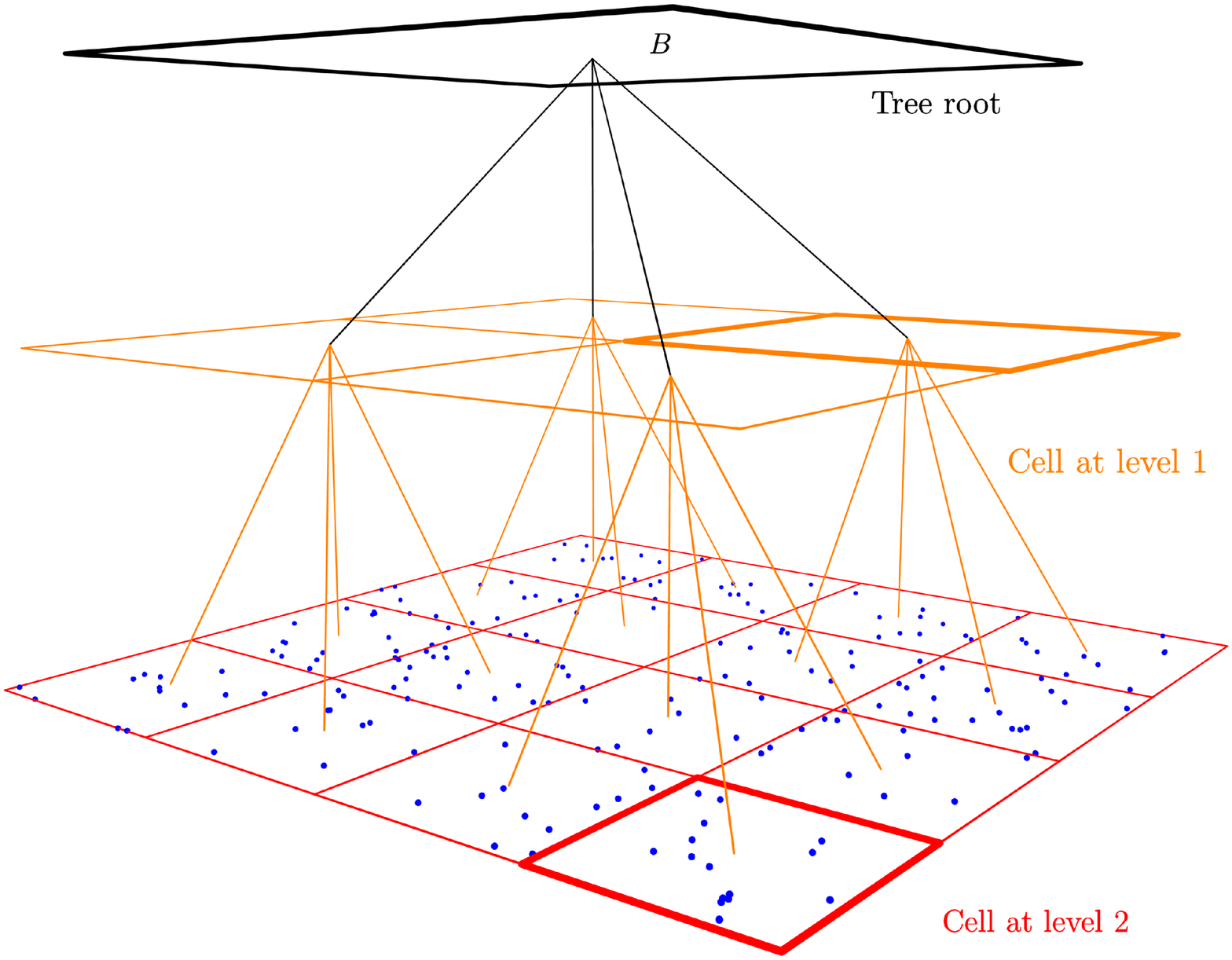}
    \caption{
    $2^d$-tree ($d=2$) representation of $B$ with particles colored in blue. 
    }
    \label{fig:octree}
\end{figure}

There mainly exist two approaches for FMMs dealing with the high-frequency Helmholtz kernel. First, the kernel-explicit methods consist in an explicit diagonalization of large far-field matrix blocks using analytic kernel expansions \cite{wideband06,ceckadarve13}, considering that the $N$-body problem in equation \ref{fmmsum} is interpreted as a matrix-vector product. 
These methods perform differently depending on the particle distribution, with complexities up to $\cc{O}(N^{3/2})$ (see \cite{ceckadarve13}). 
Second, the kernel-independent methods (see for instance \cite{darve09,yingbiroszorin04}) 
allow to use the same formalism, hence the same implementation, 
for a given class of kernels. Based on a specific definition of the \textit{well-separateness} (i.e. the subsets of $B\times B$ on which the far-field approximations can be applied), the \textit{directional} 
kernel-independent approaches exploit the low-rank property of particular blocks of a modified kernel \cite{engquistying,messner12,brandt91}. As opposed to the kernel-explicit methods, the kernel-independent directional approaches allow to derive fast algorithms with complexity $\cc{O}(N\hspace{0.05cm}\log\hspace{0.02cm}N)$ regardless of the particle distribution (see \cite{messner12}).

The directional FMM algorithms described in the literature exploit low-rank approximations to compress matrices representing the far-field contributions \cite{engquistying,messner12,messnerINRIA}. However, 
some highly efficient kernel-independent methods for low-frequency and translation-invariant 
kernels use Fast Fourier Transforms (FFTs) to efficiently process these matrices in diagonal form \cite{schoberteibert10}. These FFT-based methods exploit \textit{equispaced grids} (i.e. cartesian grids or tensorization of the same uniform sampling $d$ times). 
In FMM formulations based on polynomial interpolation, the use of FFT techniques was shown to be more efficient than low-rank approximations in the low-frequency regime \cite{blanchard:hal-01228519}, especially for high approximation orders. 
To our knowledge, directional FFT-based FMMs for the high-frequency regime have not been investigated in the past. 
Moreover, in such polynomial inter\-po\-la\-tion-\-ba\-sed FMMs, the matrices associated to some FMM operators have to be precomputed, but their number drastically increases in the high-frequency regime (see \cite{messnerINRIA}), which strongly increases the precomputation cost. FFT techniques can also help to reduce such precomputation cost.

In this paper, 
we aim at efficiently solving the N-body problem \eqref{fmmsum} with oscillatory kernels in both low- and high-frequency regimes. 
We thus present the first directional 
(hence $\cc{O}\left(N\hspace{0.05cm}log \hspace{0.05cm}N\right)$) 
interpolation-based FMM taking advantage of polynomial interpolation on equispaced grids to allow fast FFT-based evaluations and precomputations of the far-field matrices. 
A consistency proof for the polynomial Lagrange interpolation on such grids and on well-separated sets when dealing with asymptotically smooth kernels is provided. This gives a rigorous explanation of the practical convergence of our approximation process despite the well-known stability issue of such polynomial interpolation. We also show how to extend the use of some symmetries in the Fourier domain in order to minimize the number of precomputed matrices. FMMs can rely on different $2^d$-tree traversals: we adapt here a specific one (namely the dual tree traversal) to directional FMMs and show its relevance, along with a suitable $2^d$-tree data structure, for the
non-uniform particle distributions typical of BIE problems. We provide a new, 
publicly available\footnote{At: \url{https://github.com/IChollet/defmm}}, 
C++ library implementing our FMM, named as \textit{defmm} (\underline{d}irectional \underline{e}quispaced interpolation-based  \underline{fmm}) and highly optimized on one CPU core. We detail the vectorization process of our near-field direct computation using 
high-level programming, and we show how to efficiently process the numerous and small FFTs required in our method. We also improve the performance of the BLAS-based operators in the multipole and local expansion translations. Finally, we detail a comparison with a state-of-the-art directional polynomial interpo\-lation-based FMM library (namely \textit{dfmm} \cite{messner12,messnerINRIA}), exhibiting important performance gains for \textit{defmm} in all the test cases.

The paper is organized as follows. In section \ref{section_presentationFMM}, we first recall the mathematical bases of the FMMs involved in this paper.  We then provide in section \ref{section_consistency} a consistency proof of our interpolation process on equispaced grids. We detail in section \ref{section_defmm} the \textit{defmm}  
algorithmic design and we show how to extend and exploit the tree symmetries in the Fourier domain. In section \ref{section_optimizations}, we present various HPC optimizations on one CPU core for \textit{defmm}, and finally in section \ref{section_numerical} we provide numerical results, including a detailed performance 
comparison. 

\section{Presentation of Fast Multipole Methods}
\label{section_presentationFMM}
We focus in this article on the FMM formulation using polynomial interpolation and its variants \cite{darve09,schoberteibert10,messner12}, which are here briefly recalled. 
\subsection{Directional FMMs} Suppose that the kernel $G$ can be factorized into
\begin{equation}
    G(\bo{x},\bo{y}) = e^{i\kappa |\bo{x}-\bo{y}|}K(\bo{x}, \bo{y}),
\end{equation}
where $K$ is an asymptotically smooth non-oscillatory kernel. Directional approaches rely on the \textit{directional parabolic separation condition} (DPSC) \cite{engquistying}, imposing conditions on well-separated subsets of $B\times B$ (i.e. pairs of cells in practice) such that the term $G_{u}$ in the expression
\begin{equation}
\label{eq_decompGinOscillatoryNonOscillatory}
    G(\bo{x},\bo{y}) = e^{i\kappa \langle \bo{x},u\rangle}\underbrace{\left(e^{i\kappa \langle \bo{x}-\bo{y}, \frac{\bo{x}-\bo{y}}{|\bo{x}-\bo{y}|} -u \rangle }K(\bo{x}, \bo{y})\right)}_{=:G_{u}(\bo{x},\bo{y})}e^{-i\kappa \langle \bo{y},u\rangle}
\end{equation}
does not oscillate, where $u\in \bb{S}^2$ (referred to as a \textit{direction}) is a well-chosen 
approximation of $(\bo{x}-\bo{y})/|\bo{x}-\bo{y}|$ on the unit sphere. Since the only term in equation \eqref{eq_decompGinOscillatoryNonOscillatory} depending on both $\bo{x}$ and $\bo{y}$ is $G_u$, fast summation methods can be derived. The DPSC is verified when $\bo{x}\in t$, $\bo{y}\in s$, where $t,s$ are two cells such that
\begin{equation}
\label{e:DPSC}
    \kappa \hspace{0.02cm}w^2 \leq \eta\hspace{0.04cm}dist(t,s)\text{ and }t\subset C(s),
\end{equation}
with $dist(t,s) := min_{\bo{x}\in t,\bo{y}\in s}|\bo{x} - \bo{y}|$ the distance between $t$ and $s$, 
$w$ the radius\footnote{i.e. the radius of the smallest ball containing the cell.} of $s$ (assuming that $w>1$) and $C(s)$ a cone directed by $u$ with apex in the center of $s$ and aperture $\frac{\mu}{\kappa\hspace{0.02cm} w}$ \cite{engquistying}, $\eta,\mu$ being two strictly positive constants. Such pairs $(t,s)$ are said to be \textit{well-separated} in the high-frequency regime.

Because the cone aperture in the inequality \eqref{e:DPSC} decreases as the cell radius $w$ increases, more wedges are needed to compute the far field for cells close to the $2^d$-tree 
root 
than for deeper cells. Since the kernel approximations depend on the cone directions,
the number of approximations required for each target cell increases from the $2^d$-tree 
leaves to the root. To lower the far field computation cost, 
the set of directions at each tree level is chosen so that a nested property is verified \cite{engquistying,borm15}. This leads to a \textit{direction tree}: each direction at a given $2^d$-tree level $E$ is the son of a unique direction at the $2^d$-tree level $E+1$.

\subsection{Interpolation-based FMMs}
\label{subsection_interpolationbasedFMM}
The polynomial interpolation techniques for hierarchical methods \cite{gibermann01,bormlars04,darve09} rely on approximations of $G$ using Lagrange  interpolation. We refer to \cite{gascasauer00} for a description of the multivariate interpolation methods and to the Lagrange interpolation problem. In the following formula, we indicate in \cgrey{gray} the terms only appearing when combining the polynomial interpolation techniques for FMM with the directional approach \cite{messner12,messnerINRIA}. Let $\Xi_s = \{\bo{y}_1,...,\bo{y}_{\# \Xi_s}\}$ and $\Xi_t = \{\bo{x}_1,...,\bo{x}_{\# \Xi_t}\}$ (where $\# \Xi_m$, $m\in \{t,s\}$, denotes the cardinal of $\Xi_m$)
be two interpolation grids in the cells $s$ and $t$ respectively, where the pair $(t,s)$
is well-separated, we have:
\begin{equation}
  \label{eq_interpapprox}
  \begin{aligned}
    G(\bo{x},\bo{y}) &\approx \cgrey{e^{i\kappa \langle \bo{x},u\rangle}}\sum_kS_k[t](\bo{x})
    \sum_lG_{\cgrey{u}}(\bo{x}_k,\bo{y}_l)S_l[s](\bo{y})\cgrey{e^{-i\kappa \langle \bo{y},u\rangle}}\\
    & \approx \sum_k\underbrace{\left(\cgrey{e^{i\kappa \langle \bo{x},u\rangle}}S_k[t](\bo{x})
        \cgrey{e^{-i\kappa \langle \bo{x}_k,u\rangle}}\right)}_{=: S_k^{\cgrey{u}}[t](\bo{x})}\sum_lG(\bo{x}_k,\bo{y}_l)
    \underbrace{\left(\cgrey{e^{i\kappa \langle \bo{y}_l,u\rangle}}S_l[s](\bo{y})
        \cgrey{e^{-i\kappa \langle \bo{y},u\rangle}}\right)}_{=: S_l^{\cgrey{u}}[s](\bo{y})},
  \end{aligned}
\end{equation}
the polynomials $S_l[s]$ and $S_k[t]$ verifying for any $\bo{y}\in s$ and any $\bo{x}\in t$
\begin{equation*}
    S_l[s](\bo{y}) \meiff{=1}{\bo{y}=\bo{y}_l}{=0}{\bo{y}\in \Xi_s, \hspace{0.02cm}\bo{y}\neq \bo{y}_l}{\in \bb{R}}, \hspace{0.5cm}S_k[t](\bo{x}) \meiff{=1}{\bo{x}=\bo{x}_k}{=0}{\bo{x}\in \Xi_t, \hspace{0.02cm}\bo{x}\neq \bo{x}_k}{\in \bb{R}}.
\end{equation*}
Suppose that $\bo{x}\in t$, $t$ being a target cell. For any source cell $s$ well-separated from $t$, we define
\begin{equation*}
    \begin{aligned}
    p_{t,s}(\bo{x}) &:= \sum_{\bo{y}\in s\cap Y} G(\bo{x},\bo{y})q(\bo{y})\\
    &\approx  \sum_kS_k^{\cgrey{u}}[t](\bo{x})\sum_lG(\bo{x}_k,\bo{y}_l)\sum_{\bo{y}\in s\cap Y}S_l^{\cgrey{u}}[s](\bo{y})q(\bo{y}),
    \end{aligned}
\end{equation*}
where the second line is obtained thanks to the approximation \eqref{eq_interpapprox}. This interpolation process can be repeated on $t'\in Sons(t)$ and $s'\in Sons(s)$ recursively, leading in the end to a multilevel algorithm whose operators, named as in \cite{darve09,yokota18}, are the following:
\begin{align*}
&\textrm{P2M~(particles-to-multipole):} & \cc{M}_{s'}^{\cgrey{u}}(\bo{y}_r') &:=\displaystyle\sum_{\bo{y}\in s\cap Y}S_r^{\cgrey{u}}[s'](\bo{y})q(\bo{y}), & &\forall \bo{y}_r'\in \Xi_{s'};\\
&\textrm{M2M~(multipole-to-multipole):} & \cc{M}_s^{\cgrey{v}}(\bo{y}_l) &:=\displaystyle\sum_{r}S_l^{\cgrey{v}}[s](\bo{y}_r')\cc{M}_{s'}^{\cgrey{u}}(\bo{y}_r'), & &\forall \bo{y}_l\in \Xi_{s};\\
&\textrm{M2L~(multipole-to-local):} & \cc{L}_t^{\cgrey{v}}(\bo{x}_k) &:= \displaystyle\sum_lG(\bo{x}_k,\bo{y}_l)\cc{M}_s^{\cgrey{v}}(\bo{y}_l), & &\forall \bo{x}_k\in \Xi_{t};\\
&\textrm{L2L~(local-to-local):} & \cc{L}_{t'}^{\cgrey{u}}(\bo{x}_h') &:= \displaystyle\sum_{k}S_k^{\cgrey{v}}[t](\bo{x}_h')\cc{L}_{t'}^{\cgrey{v}}(\bo{x}_k), & &\forall \bo{x}_h'\in \Xi_{t'};\\
&\textrm{L2P~(local-to-particles):} & p_{t,s}(\bo{x}) &\approx\displaystyle\sum_hS_h^{\cgrey{u}}[t'](\bo{x})\cc{L}_{t'}^{\cgrey{u}}(\bo{x}_h'), & &\forall \bo{x}\in X\cap t;
\end{align*}
where: $\Xi_{s'} = \{\bo{y}_1',...,\bo{y}_{\# \Xi_{s'}}'\}$ and $\Xi_{t'} = \{\bo{x}_1',...,\bo{x}_{\# \Xi_{t'}}'\}$ are the interpolation grids in $s'$ and $t'$ respectively; $v$ is a son of $u$ in the direction tree such that $u$ and $v$ are the best approximations of $(ctr(t)-ctr(s))/|ctr(t)-ctr(s)|$; 
$ctr(c)$ denoting the center of the cell $c$. The direct evaluation of $p_{t,s}$ for all $\bo{x}\in X\cap t$ corresponds to the P2P  (particles-to-particles) operator between $t$ and $s$. The result of a P2M or M2M operation 
is a \textit{multipole} expansion and the result of a M2L or L2L operation is a \textit{local} expansion. The set of cells $s$ such that $p_{t,s}$ is approximated (i.e. not directly evaluated through a P2P operator) is named the \textit{interaction list} of $t$.

All the FMM operators we defined can be interpreted as matrix-vector products, where the vectors are local or multipole expansions. Hence, we shall for example use in the remainder 
\textit{M2L matrices} to refer to the matrices involved in the M2L matrix-vector product (and similarly for the other operators).

\subsection{Equispaced grids}
\label{subsection_equispacedgrids}
When dealing with multivariate polynomial interpolations, the most natural approach relies on tensorizations
of 1D interpolation grids in a cuboid \cite{darve09,bormlars04,gibermann01}.
Denoting $B_\star = \lbrack 0,1\rbrack^d$ the reference unit box, any
cell $c$ 
writes as 
\begin{equation}\label{CellCorrespondance}
  c = \boldsymbol{\alpha}_c + \beta_c \cdot B_\star
\end{equation}
where $\beta_c = \mathrm{diam}(B)/2^{L}$ with $L = \mathrm{level}(c)$, and $\boldsymbol{\alpha}_c\in \mathbb{R}^d$.
We focus here on equispaced grids, which allow to take advantage of the translation
invariance of the kernel, i.e. $G((\bo{x}+\bo{z})-(\bo{y}+\bo{z})) = G(\bo{x} - \bo{y})$. The interpolation grid
$\Xi_c\subset\mathbb{R}^d$ in cell $c$ is obtained from a reference grid $\Xi_\star\subset \mathbb{R}^d$ of $B_\star$
through \eqref{CellCorrespondance} 
\begin{equation}
  \begin{aligned}
    & \Xi_c = \boldsymbol{\alpha}_c + \beta_c \cdot \Xi_\star\\[3pt]
    & \text{where}\quad \Xi_\star := \mathbb{L}\times \dots \times\mathbb{L}\\
    & \textcolor{white}{where}\quad \mathbb{L} := \{ \ell/L,\; \ell = 0,\dots ,L\}.
  \end{aligned}
\end{equation}
These transformations can be used to transfer all evaluations 
of the kernel $G$ on the reference grid $\Xi_\star$.
Take a source cell $s$ and a target cell $t$, and assume that they belong to the same level of their $2^d$-tree
and thus admit the same size and $\beta_s = \beta_t$.
Any $\bo{x}\in s$ can be written as $\bo{x} = \boldsymbol{\alpha}_s + \beta_s\hat{\bo{x}}$ and any $\bo{y}\in t$
can be expressed as $\bo{y} = \boldsymbol{\alpha}_t + \beta_t\hat{\bo{y}}$, so that we have
\begin{equation}\label{NormalizedKernel}
  \begin{aligned}
    & G(\bo{x}-\bo{y}) = \mathcal{G}_{s,t}(\hat{\bo{x}} - \hat{\bo{y}})\\
    & \text{with}\quad \mathcal{G}_{s,t}(\bo{z}):= G(\boldsymbol{\alpha}_s-\boldsymbol{\alpha}_t + \beta_s \bo{z}).
  \end{aligned}
\end{equation}
The function $\bo{x},\bo{y}\mapsto G(\bo{x}-\bo{y})$ induces a linear map
$G(\Xi_t,\Xi_s):\mathbb{C}[\Xi_s]\to \mathbb{C}[\Xi_t]$. Similarly
$\bo{p},\bo{q}\mapsto \mathcal{G}_{s,t}(\bo{p}-\bo{q})$ induces a linear map 
$\mathcal{G}_{s,t}(\Xi_\star,\Xi_\star): \mathbb{C}[\Xi_\star]\to \mathbb{C}[\Xi_\star]$
and, by construction, these two linear maps can be identified to the same matrices
under consistent orderings $G(\Xi_t,\Xi_s) = \mathcal{G}_{s,t}(\Xi_\star,\Xi_\star)$.
Rewriting the kernel as above, we explicitly take advantage of the fact that $\beta_s = \beta_t$.
When examining interactions between $s$ and $t$, the normalized kernel $\mathcal{G}_{s,t}$ is evaluated
over the difference grid
\begin{equation}
  \begin{aligned}
    \Xi_{\sharp}
    & := \{ \bo{p}-\bo{q},\; \bo{p},\bo{q}\in \Xi_{\star}\}\subset \mathbb{R}^d\\
    & := \{ \ell/L,\; \ell\in\mathbb{N},\; -L\leq \ell\leq +L\}^d.
  \end{aligned}
\end{equation}
Obviously $\Xi_{\star}\subset \Xi_{\sharp}$, and any vector $u\in \mathbb{C}[\Xi_\star]$ can be considered as
an element of $\mathbb{C}[\Xi_\sharp]$ by means of extension by zero: define $\chi(u)\in \mathbb{C}[\Xi_\sharp]$
by $\chi(u)(\bo{p}) := u(\bo{p})$ if $\bo{p}\in \Xi_\star$ and $\chi(u)(\bo{p}) := 0$ otherwise. The linear map $\chi$
simply reduces to a matrix with entries equal to $0$ or $1$, with at most one non-vanishing entry per row, so
applying such a matrix to a vector is computationally harmless.
The application of this matrix numerically corresponds to a \textit{zero-padding}.

We have $\mathrm{card}(\Xi_{\sharp}) = T^d$ with $T:=(2L+1)$, i.e. $\Xi_{\sharp}$ contains $T$ samples in each
direction. The optimization we wish to exploit is based on Fast Fourier Transforms (FFTs), 
this is why we embed the action of \eqref{NormalizedKernel}
in a periodic setting. The kernel $\mathcal{G}_{s,t}$ can be extended by $T/L$-periodicity, denoting  $\mathcal{G}_{s,t}^\sharp$
the unique element of $\mathbb{C}[\Xi_{\sharp} + \mathbb{Z}^d]$ satisfying
\begin{equation}\label{CirculantKernel}
  \begin{aligned}
     &\mathcal{G}_{s,t}^\sharp(\bo{p} + (T/L)\boldsymbol{\tau}) = \mathcal{G}_{s,t}(\bo{p})\\
    & \text{for all}\quad \bo{p}\in \Xi_\sharp, \boldsymbol{\tau}\in\mathbb{Z}^d.
  \end{aligned}
\end{equation}
In particular $\mathcal{G}_{s,t}^\sharp(\bo{p}) = \mathcal{G}_{s,t}(\bo{p})$ for $\bo{p}\in \Xi_\sharp$.
Denote $\mathcal{G}_{s,t}^\sharp(\Xi_\sharp,\Xi_\sharp)$ the matrix associated to the
mapping from $\mathbb{C}[\Xi_\sharp]\to \mathbb{C}[\Xi_\sharp]$ associated to $p,q\mapsto\mathcal{G}_{s,t}^\sharp(\bo{p}-\bo{q})$. This matrix admits a circulant form due to the periodicity of
the corresponding kernel. The notations introduced above lead to a factorization of
M2L interactions between cells $s$ and $t$,
\begin{equation}\label{FactorizedM2LMatrix1}
  G(\Xi_s,\Xi_t) = \chi^\top\cdot\mathcal{G}_{s,t}^\sharp(\Xi_\sharp,\Xi_\sharp)\,\cdot\chi.
\end{equation}
As we are considering tensorized interpolation grids, and since $\mathcal{G}_{s,t}^\sharp(\Xi_\sharp,\Xi_\sharp)$
admits circulant form, we naturally make use of a multi-dimensional Discrete Fourier Transform (DFT)
to process this matrix. Let us briefly explain how.
Given the dual grid $\widehat{\Xi}_\sharp = L\cdot\Xi_\sharp
\subset \mathbb{R}^d$, the $d$-dimensional DFT refers to the linear maps $\mathbb{F}:\mathbb{C}[\Xi_\sharp]\to
\mathbb{C}[\widehat{\Xi}_\sharp]$ defined by the formula
\begin{equation}\label{DFTformula}
  \mathbb{F}(u)(\boldsymbol{\xi}) := \frac{1}{T^{d/2}}\sum_{\bo{x}\in \Xi_\sharp}u(\bo{x})
  \exp(-2\imath \pi\, \boldsymbol{\xi}\cdot \bo{x}) \quad\quad \boldsymbol{\xi}\in \widehat{\Xi}_\sharp.
\end{equation}
The inverse DFT is simply the adjoint map defined by the formula $\mathbb{F}^*(\hat{u})(\bo{x}) =
T^{-d/2}\sum_{\bo{\xi}\in \widehat{\Xi}_\sharp}\hat{u}(\bo{\xi}) \exp(+2\imath \pi\, \bo{\xi}\cdot \bo{x})$ for
$\bo{x}\in\Xi_\sharp$. Due to the circulant form of the matrix $\mathcal{G}_{s,t}^\sharp(\Xi_\sharp,\Xi_\sharp)$,
there exists a diagonal matrix $\mathbb{D}_{s,t}$ such that 
\begin{equation}\label{FourierDiagonal}
  \begin{aligned}
    & \mathcal{G}_{s,t}^\sharp(\Xi_\sharp,\Xi_\sharp) = \mathbb{F}^*\cdot \mathbb{D}_{s,t}\cdot\mathbb{F}\\
    & \text{hence}\quad G(\Xi_s,\Xi_t) = \chi^\top\cdot\mathbb{F}^*\cdot \mathbb{D}_{s,t}\cdot\mathbb{F}\cdot\chi.
  \end{aligned}
\end{equation}
The diagonal of $\mathbb{D}_{s,t} = \mathrm{diag}_{\bo{\xi}\in \widehat{\Xi}_\sharp }
(\,\mathbb{F}(\mathcal{G}_{s,t})(\bo{\xi})\,)$ contains the Fourier coefficients of
the kernel. From the factorized form above, we see that matrix-vector
product for the M2L matrix $G(\Xi_s,\Xi_t)$ can be processed efficiently. Indeed
the action of $\chi$ is inexpensive, and $\mathbb{F}$ can be applied efficiently with
$\cc{O}(d\, L^d\, \log L )$ complexity by means of FFTs. Notice that the diagonal of $\bb{D}_{s,t}$ (i.e. its non-zero entries) is obtained by DFT of the first column of $\mathcal{G}_{s,t}^\sharp(\Xi_\sharp,\Xi_\sharp)$ and can also be efficiently processed using FFT.

Since the same Fourier basis is used for all M2L matrices, one can apply $\bb{F}$ and $\bb{F}^*$ (as well as $\chi$ and $\chi^\top$) once to each multipole and local expansions respectively \textit{before} and \textit{after} all the M2L evaluations.
This reduces M2L matrix evaluations to products with diagonal matrices $\bb{D}_{s,t}$ (i.e. Hadamard products), each with a complexity of $(2L-1)^d$ flop.  
In comparison with the product of a vector by a rank-$k$ matrix approximation of a M2L matrix, requiring $2kL^d$ operations,
the product with the diagonalization of the circulant embedding of this matrix is \textit{theoretically} less costly if $k\geq 2^{d-1}$,
which is valid except when requesting very low overall accuracies. 
However, low-rank approaches usually rely on highly optimized implementations (namely 
BLAS routines) to significantly reduce the computation times of the dense matrix-vector (or matrix-matrix) products (see e.g. \cite{messnerINRIA}). 
Therefore, theoretical operation counts are insufficient 
to compare a FFT-based approach with a low-rank one. We will have to compare practical time measurements using
implementations of both approaches to determine the best one.  
One can already notice that, for non-oscillatory kernels, 
FFT-based approaches offer better or similar
performance (depending on the interpolation order) than low-rank compressions for the M2L matrix evaluations
\cite{blanchard:hal-01228519,chenaubryoppelstruparsenlisdarve18}: 
this encourages us to investigate the FFT-based approaches for the oscillatory kernels. 

In this purpose, we recall first that 
the interpolation process on equispaced 
grids is known to be subject to 
the Runge phenomenon, which may cause the process to diverge, especially for high interpolation orders (see for instance \cite{chenaubryoppelstruparsenlisdarve18}).

\section{Consistency of the equispaced interpolation}
\label{section_consistency}
Despite the asymptotic instability of the Lagrange polynomial interpolation process on equispaced grids, we 
show in this section that the approximation \eqref{eq_interpapprox} actually converges on well-separated sets. 
We define the \textit{Multipole Acceptance Criterion} (or MAC) $\cc{A}$ as a boolean function that takes two cells in argument and returns $1$ if and only if these two cells are well-separated. As in \cite{darve09,blanchard:hal-01228519}, we only consider cells at the same $2^d$-tree level as input parameters of the MAC. These cells $t,s$ 
thus have the same radius $a$, i.e. they are translations of $[-a,a]^d$ in $\bb{R}^d$ (see section \ref{subsection_equispacedgrids}). The interpolation grids in $t,s$ can be expressed relatively to the cell centers, so the polynomial interpolation of $G$ on $t\times s$ can be seen as an interpolation of $\tilde{G}$ on $[-a,a]^{2d}$ using $\tilde{G}(\bo{x},\bo{y}) := G(ctr(t)+\bo{x},ctr(s)+\bo{y})$. Here, $\tilde{G}$ depends on $t\times s$.

We now introduce and prove the following theorem, regarding the interpolation process consistency on our equispaced grids.
\begin{theorem}
\label{theo_consistancy_ufmm}
Suppose that the condition $\cc{A}(t,s) = 1$ implies that $\tilde{G}$ is analytic in each variable at any point in $[-a,a]$ with a convergence radius $R$ such that $R>\frac{2a}{e}$. Then the Lagrange interpolation of $\tilde{G}$ on $t\times s$ using $L$ equispaced points in each variable, denoted by $\cc{I}_L^{t\times s}[\tilde{G}]$, verifies
\begin{equation*}
  \displaystyle\mathop{lim}_{L\rightarrow +\infty}\big|\big|\cc{I}_L^{t\times s}[\tilde{G}] -\tilde{G} \big|\big|_{L^\infty(t\times s)} = 0.
  \end{equation*}
\end{theorem}

      \subsection{Preliminary results}
      \label{subsection_dufmm_preliminaryresults}
      The idea of the proof of Theorem \ref{theo_consistancy_ufmm} consists in combining 1D convergence estimates with the interpolation error results provided in \cite{mossnerreif09} (restricted to our asymptotically smooth kernel functions). Our convergence proof differs from the one of \cite{darve09,messner12} because we cannot rely on the properties of the Chebyshev grids.  We denote by $\scr{C}^\infty(\Omega)$ the space of multivariate functions $f$ with bounded derivatives $\partial^{\beta}f := \partial^{\beta_1}f...\partial^{\beta_{2^d}}f$ on $\Omega$, $\forall \beta\in \bb{N}^{2d}$ such that $\displaystyle\mathop{max}_{k=1,...,2d}\beta_{k}\leq L$ for any domain $\Omega\subset \bb{R}^{2d}$. We denote by $\kk{A} := \{\alpha=(\alpha_1,...,\alpha_{2d})\in \{0,1\}^{2d},\hspace{0.02cm}\exists \hspace{0.02cm}j, \hspace{0.02cm}\alpha_j = 1\}$. We also use the notation $||f||_\infty := \mathop{sup}_{\bo{z}\in [-a,a]^{2d}}|f((\bo{z})|$.
      
      \begin{theorem}
        \label{stability_theo_multivariate}
        (\cite{mossnerreif09} Thm. 2.1) For $f\in \scr{C}^\infty([-a,a]^{2d})$, the product interpolation $\cc{I}^{[-a,a]^{2d}}_L[f]$ of $f$ in $[-a,a]^{2d}$ with the same 1D rule with $L$ interpolation nodes in each variable verifies
        \begin{equation*}
          \bigg|\bigg|f-\cc{I}^{[-a,a]^{2d}}_L[f]\bigg|\bigg|_{\infty} \leq \displaystyle\sum_{\alpha\in \kk{A}}\omega_L^{\bar{\alpha}}\Big|\Big|\partial^{\alpha L}f\Big|\Big|_{\infty}
        \end{equation*}
        where $\alpha L = (\alpha_1 L, ..., \alpha_{2d}L)$, $\bar{\alpha}:=\displaystyle\sum_{k=1}^{2d}\alpha_k$ and $\omega_L := \frac{1}{L!}\Big|\Big| \displaystyle\prod_{k=0}^{L-1}(\cdot - x_k) \Big|\Big|_{\infty}$, $x_k$ being the $k^{th}$ interpolation point of the 1D rule.
      \end{theorem}
      For equispaced grids, the constant $\omega_L$ can be bounded using the following lemma.
      \begin{lemma}
            \label{stability_lemma_omega}
            Let $\{x_k:=-a + 2ak/(L-1)\hspace{0.1cm}|\hspace{0.1cm}k\in [\![0,L-1]\!]\}\subset [-a,a]$ an interpolation grid of equispaced points. We have $\omega_L \leq \left(\frac{2a}{L-1}\right)^L/(4L).$
      \end{lemma}

      \begin{proof}        
        Let $x\in [-a,a]$. We have $
            \Big|\displaystyle\prod_{j=0}^{L-1}(x - x_j)\Big| = (2a)^L\Big|\displaystyle\prod_{j=0}^{L-1}\left(\frac{x+a}{2a} - \frac{j}{(L-1)}\right)\Big|$.
        Let $y := \left(\frac{x+a}{2a}\right) \in [0,1]$. This leads to $\Big|\displaystyle\prod_{j=0}^{L-1}(x - x_j)\Big| \leq \left(\frac{2a}{L-1}\right)^L\Big|\displaystyle\prod_{j=0}^{L-1}y(L-1) - j\Big|$. Because $y(L-1) \in [0,(L-1)]$, $\Big|\displaystyle\prod_{j=0}^{L-1}y(L-1) - j\Big|$ is maximal for $y\in (0,1/(L-1))$. Using a simple recurrence, one may easily show that $\displaystyle\prod_{j=0}^{L-1}\Big|y(L-1) - j\Big| \hspace{0.2cm}\leq \frac{(L-1)!}{4}$,
        which implies that $
            \omega_L \leq \frac{1}{L!}\left(\frac{2a}{L-1}\right)^L(L-1)!/4\hspace{0.2cm}\leq \left(\frac{2a}{L-1}\right)^L/(4L)$.
      \end{proof}

      We now want to bound the partial derivatives of the interpolated function. This is the purpose of the following lemma.
      \begin{lemma}
        \label{consistancy_lemma_cauchy}
	If $f$ is analytic in all its variables at any point of $[-a,a]$ with a convergence radius $R>0$, we have $||\partial^{\alpha L}f||_\infty \leq \frac{C}{r^{2d}}\inpar{\frac{L!}{r^L}}^{\bar{\alpha}}$ 
	with $0 < r < R$, $\bar{\alpha}:=\sum_{k=1}^{2d}\alpha_k$, $C\in \bb{R}^{*+}$ being a constant independent of $L$, and $\alpha$ being defined as in Thm. \ref{stability_theo_multivariate}.
      \end{lemma}
      
      \begin{proof}
	Since $f$ is analytic in all its variables, we can apply the Cauchy integral formula (see \cite{lars79}, Chapter 2, Theorem 6), allowing us to write:
	\begin{equation*}
	  \begin{aligned}
	    f(\bo{p}) &= \left(\frac{1}{2\pi i}\right)^{2d}\int_{\Gamma_1} ...\int_{\Gamma_{2d}} \frac{f(\bo{z})}{(z_1-p_1)...(z_{2d}-p_{2d})} dz_1 ...dz_{2d},
	  \end{aligned}
	\end{equation*}
	where $\Gamma_j:=\{\xi\in \bb{C}\hspace{0.1cm}|\hspace{0.1cm}dist(\xi,[-a,a]=r)\}$.
	We thus have:
	\begin{equation*}
	  \begin{aligned}
	    \Big|\Big|\partial^{\alpha L}_{\bo{p}}f(\bo{p})\Big|\Big|_\infty &\leq \left(2\pi\right)^{-2d}\int_{\Gamma_1} ...\int_{\Gamma_{2d}} f(\bo{z})\Big|\Big|\partial_{\bo{p}}^{\alpha L}\frac{1}{(z_1-p_1)...(z_{2d}-p_{2d})}\Big|\Big|_\infty dz_1 ...dz_{2d}\\
	    &\leq \frac{|\Gamma_1| ...|\Gamma_{2d}|}{\left(2\pi\right)} \left( \mathop{sup}_{\bo{z}\in\Gamma_1\times ...\times \Gamma_{2d}}\Big|f(\bo{z})\Big|\right) \Big|\Big|\partial_{\bo{p}}^{\alpha L}\frac{1}{(z_1-p_1)...(z_{2d}-p_{2d})}\Big|\Big|_\infty.
	  \end{aligned}
	\end{equation*}
	The term $\left( \mathop{sup}_{\bo{z}\in\Gamma_1\times ...\times \Gamma_{2d}}|f(\bo{z})|\right)$ is bounded thanks to the analyticity of $f$ on a neighborhood of $[-a,a]^{2d}$ encompassing the $\Gamma_j$'s (with a convergence radius equal to $R$). Hence, there exists a constant $M(\Gamma_1,...,\Gamma_{2d})\in \bb{R}^{+*}$ such that, defining 
	\begin{equation*}
	    C(\Gamma_1,...,\Gamma_{2d}) := \left(2\pi\right)^{-2d}|\Gamma_1| ...|\Gamma_{2d}|M(\Gamma_1,...,\Gamma_{2d}),
	\end{equation*}
	where $|\Gamma_j |$ denotes the length of the path $\Gamma_j$, $j\in [\![1,2d]\!]$, with $\bar{\alpha}:=\displaystyle\sum_{k=1}^{2d}\alpha_k$, we have
	\begin{equation*}
	  \begin{aligned}
	    \Big|\Big|\partial^{\alpha L}f(\bo{p})\Big|\Big|_\infty &\leq C(\Gamma_1,...,\Gamma_{2d}) \Big|\Big|\partial^{\alpha L}\frac{1}{(z_1-p_1)...(z_{2d}-p_{2d})}\Big|\Big|_\infty\\
	    (\textit{Since }\alpha_k=0,1)&\leq C(\Gamma_1,...,\Gamma_{2d}) \displaystyle \prod_{k=1}^{2d}\left( L!^{\alpha_k}\Big|\Big|(z_k-p_k)\Big|\Big|_\infty^{-L\alpha_k-1} \right)\\
	    &\leq C(\Gamma_1,...,\Gamma_{2d})\underbrace{\displaystyle \prod_{k=1}^{2d}\left( L!^{\alpha_k}r^{-L\alpha_k-1} \right)}_{=\inpar{\frac{L!}{r^L}}^{\bar{\alpha}}/(r^{2d})}.
	  \end{aligned}
	\end{equation*}
        ~
        
    \end{proof} 

      \subsection{Proof of the main theorem}
      \label{subsection_dufmm_proofofthemaintheorem}
      We can now prove Thm. \ref{theo_consistancy_ufmm}.
      \begin{proof}
        Following Lem. \ref{stability_lemma_omega} and Thm. \ref{stability_theo_multivariate}, we have,
        \begin{equation*}
          \begin{aligned}
            \Big|\Big|\cc{I}_L^{t\times s}[\tilde{G}] - \tilde{G}\Big|\Big|_{\infty} &\leq \sum_{\alpha\in \kk{A}}\left(\prod_{k=1}^{2d}\omega_L^{\alpha_k}\right)\Big|\Big|\partial^{\alpha L}\tilde{G}\Big|\Big|_{\infty}\\
            &\leq \sum_{\alpha\in \kk{A}}
            \prod_{k=1}^{2d}\left(\left( \frac{2a}{L-1}\right)^L/(4L)\right)^{\alpha_k}
            \Big|\Big|\partial^{\alpha L}\tilde{G}\Big|\Big|_{\infty}
          \end{aligned}
        \end{equation*}
        which becomes, thanks to Lem. \ref{consistancy_lemma_cauchy}
        \begin{equation*}
          \begin{aligned}
            \Big|\Big|\cc{I}_L^{t\times s}[\tilde{G}] - \tilde{G}\Big|\Big|_{\infty}
            &\leq C\sum_{\alpha\in \kk{A}}
            \prod_{k=1}^{2d}\left(\left( \frac{2a}{L-1}\right)^L
            \frac{L!}{4r^LL}\right)^{\alpha_k}r^{-1}.
          \end{aligned}
        \end{equation*}
        Now, by applying Stirling's inequality $L!\leq e^{-(L-1)}L^{L+1/2}$, one obtains
        \begin{equation*}
          \begin{aligned}
            \Big|\Big|\cc{I}_L^{t\times s}[\tilde{G}] - \tilde{G}\Big|\Big|_{\infty}
            &\leq C\sum_{\alpha\in \kk{A}}
            \prod_{k=1}^{2d}\left(\left( \frac{2a}{L-1}\right)^L
            \frac{e^{-(L-1)}L^{L+1/2}}{4r^LL}\right)^{\alpha_k}r^{-1}\\
            &\leq C\sum_{\alpha\in \kk{A}}
            \prod_{k=1}^{2d}\left(\left( \frac{2a}{L-1}\right)^L
            \frac{eL^{L+1/2}}{4(er)^LL}\right)^{\alpha_k}r^{-1}\\
            &\leq C\sum_{\alpha\in \kk{A}}
            \prod_{k=1}^{2d}\left(\left( \frac{2a}{re}\right)^L
            \left(\frac{L}{L-1}\right)^L\frac{e}{4\sqrt{L}}\right)^{\alpha_k}r^{-1}.
          \end{aligned}
        \end{equation*}
        For $L\geq 2$, we have $\left(\frac{L}{L-1}\right)^L \leq 4$, which allows to write
        \begin{equation*}
          \begin{aligned}
            \Big|\Big|\cc{I}_L^{t\times s}[\tilde{G}] - \tilde{G}\Big|\Big|_{L^{\infty}([-a,a]^{2d})}
            &\leq C\sum_{\alpha\in \kk{A}}
            \prod_{k=1}^{2d}\left(\left( \frac{2a}{re}\right)^L
            \frac{e}{\sqrt{L}}\right)^{\alpha_k}r^{-1}\\
            &\leq Cr^{-2d}\sum_{\alpha\in \kk{A}}
            \underbrace{\prod_{k=1}^{2d}\left(\left( \frac{2a}{re}\right)^L\frac{e}{\sqrt{L}}\right)^{\alpha_k}}_{=\inpar{\frac{2a}{re}}^{\bar{\alpha} L}\inpar{\frac{e}{\sqrt{L}}}^{\bar{\alpha}}}
          \end{aligned}
        \end{equation*}
        using $\bar{\alpha}=\displaystyle\sum_{k=1}^{2d}\alpha_k$. 
        The number of terms in the sum over $\kk{A}$ 
        is finite and depends only on the dimension: there is indeed  $2^{2d}-1$ terms in this sum since $\alpha=(0,...,0)$ does not verify $||\alpha ||_\infty=1$. 
        This estimate thus tends to zero when $L$ tends to infinity if $\frac{2a}{re} < 1$. Since $r < R$, this is verified if $\frac{2a}{e} < R$. Indeed, $\inpar{\frac{e}{\sqrt{L}}}^{\bar{\alpha}} \leq \inpar{\frac{e}{\sqrt{2}}}^{2d}$ since we assumed that $L \geq 2$. In addition, each $\inpar{\frac{2a}{re}}^{\bar{\alpha} L}$ can be bounded by $\inpar{\frac{2a}{re}}^{L}$ since $\frac{2a}{re} < 1$. We finally have
        \begin{equation}
          \label{convergence_estimate_dufmm}
          \begin{aligned}
            \Big|\Big|\cc{I}_L^{t\times s}[\tilde{G}] - \tilde{G}\Big|\Big|_{L^{\infty}([-a,a]^{2d})}
            &\leq \underbrace{\inpar{Cr^{-2d}(2^{2d}-1)\inpar{\frac{e}{\sqrt{2}}}^{2d}}}_{\text{Does not depend on }L}\underbrace{\left( \frac{2a}{re}\right)^L}_{\displaystyle\mathop{\rightarrow}_{L\rightarrow +\infty}0}.
          \end{aligned}
        \end{equation}
      \end{proof}
      
      This proof has a geometric interpretation. In the inequality (\eqref{convergence_estimate_dufmm}), the term $\left( \frac{2a}{re}\right)^L$ somehow corresponds to a MAC such as in \cite{denen02}: $a$ refers to the radius of an interacting cell and $r$ is related to the distance between the interacting cells. 
      The greater this distance, the greater $r$ and the better this estimate. Another information we obtain from the inequality (\eqref{convergence_estimate_dufmm}) is that the convergence should be geometric in the 1D interpolation order. Indeed, considering $\theta := \frac{radius(t)+radius(s)}{dist(t,s)}$, 
      Estimate (\ref{convergence_estimate_dufmm}) somehow indicates that the error of the interpolation process is $\cc{O}\left(\left(\frac{\theta}{e}\right)^L\right)$. In practice, our new directional MAC (defined in section \ref{sss:DTT} below), as well as the MAC from \cite{darve09,messner12,messnerINRIA} (referred to as the \textit{strict} MAC\footnote{Two cells at the same $2^d$-tree level comply with the strict MAC if the distance between them is greater or equal than their side length.} \label{def:strictMAC} in this article), both verify the assumptions of theorem \ref{theo_consistancy_ufmm}: the interpolation process of $G$ therefore converges on such well-separated sets.
      
      We conclude this section by two important remarks. First, the practical efficiency of the interpolation-based FMM relies on the implicit assumption that the constant $C(\Gamma_1,...,\Gamma_{2d})$ is small. This is the case when using the strict MAC in the low-frequency regime (a justification can be found in \cite{messner12}) 
      and our new directional MAC (see section \ref{sss:DTT}) in the high-frequency one. 
      Second, the ill-conditioning of such interpolation on equispaced grids may cause an exponential amplification of floating-point rounding errors at  
      the boundaries of the interpolation domain (see \cite{smith06,plattetrefethenkuijlaars09}). We thus cannot expect a numerical convergence for 
      any order in practice. Nevertheless, according to our tests (not shown here), we can reach a relative error of 
      $10^{-12}$ on 3D particle distributions with double-precision arithmetic before facing 
      numerical instabilities. Practical applications usually require much less accurate approximations.

\section{\textit{defmm}: a directional equispaced interpolation-based FMM}
\label{section_defmm}
In this section, we present our new \textit{defmm} library,  
integrating polynomial  interpolations  on  equispaced  grids within a directional kernel-independent FMM.

\subsection{Directional FFT-based FMM}
The integration of FFT techniques exploiting equispaced grids in a directional polynomial interpolation-based FMM requires operations between various directional expansions to be converted in the Fourier domain. 

\subsubsection{M2F and F2M operators}
As described in section \ref{section_presentationFMM}, 
the M2L application on equispaced grids 
involves three steps: the conversion of the involved multipole expansion (extended by zero-padding)
into the Fourier domain, the application of the diagonal M2L operator in this domain, and then the conversion of the resulting (extended) local expansion in the Fourier domain into a local expansion. 
We thus introduce two extra operators, the M2F (multipole-to-Fourier) and F2L (Fourier-to-local) ones, applying the matrices $\bb{F}\chi$ to multipole expansions in the real domain and $\chi^T\bb{F}^*$ to local expansions in the Fourier domain. Let $T$ be a M2L matrix, we have (see section \ref{subsection_equispacedgrids})
\begin{equation*}
    T = \underbrace{\inpar{\chi^\top\bb{F}^*}}_{\text{F2L}}\hspace{0.2cm}\underbrace{\bb{D}_{s,t}}_{\text{Diagonal M2L}} \hspace{0.2cm}\underbrace{\inpar{\bb{F}\chi}}_{\text{M2F}}.
\end{equation*} 

In Figure \ref{fig_linkbetweenexpansions}, we depict how the different operators are applied 
in \textit{defmm}. The operators are 
non-directional (i.e. without the terms in \cgrey{grey} in section \ref{subsection_interpolationbasedFMM}) in the low-frequency regime, and directional in the high-frequency one.  

\begin{figure}[t]
    \centering
    \includegraphics[width=0.8\linewidth]{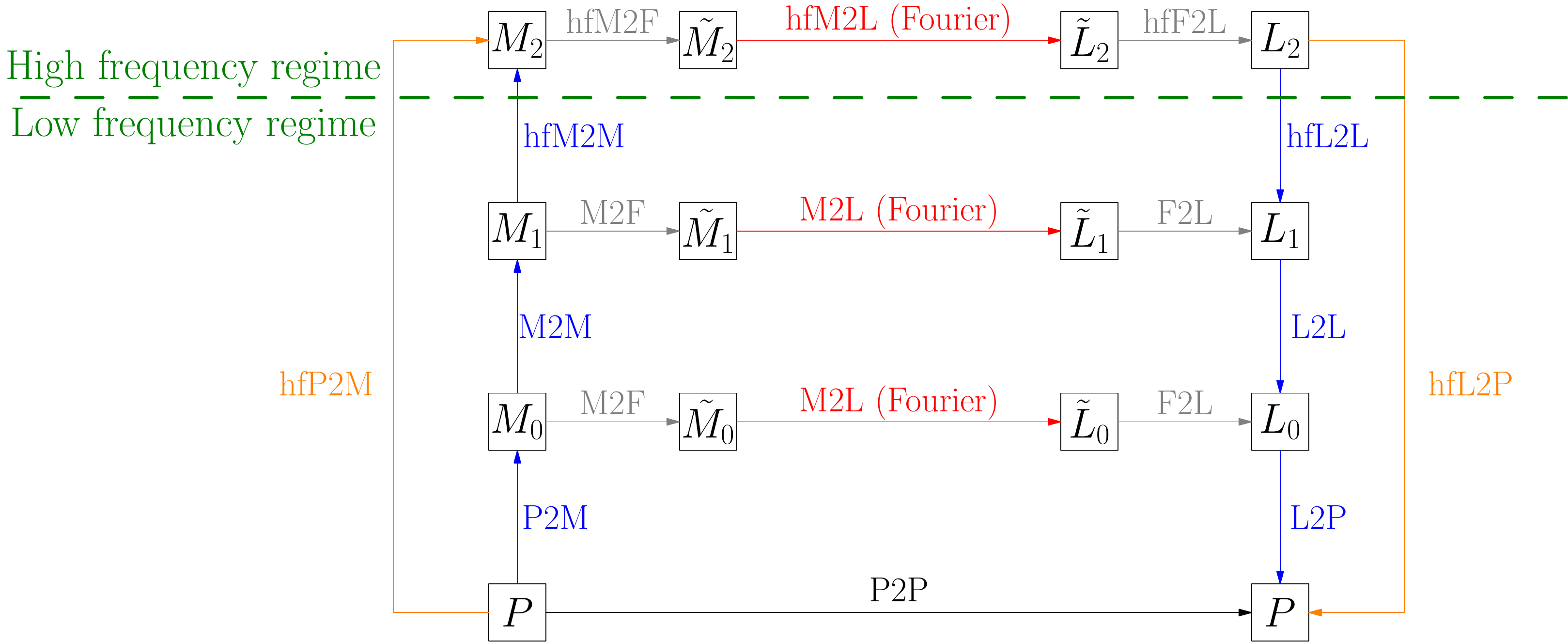}
    \caption{\label{fig_linkbetweenexpansions}Operators and expansions in \textit{defmm}. \texttt{hf} is added in front of the directional operators (used in the high-frequency regime). P: particles; 
    $M_i$: multipole expansions; $L_i$: local expansions; $\tilde{M}_i$,$\tilde{L}_i$: multipole and local expansions in the Fourier domain. 
    }
\end{figure}

\subsubsection{Direction generation}
\label{subsubsection_directiongeneration}
In the directional FMM presented in \cite{engquistying} the symmetries within 
the sets of directions at each tree level are exploited to reduce the number of precomputed M2L matrices. 
This however constrains the  way the directions are generated. On the contrary, using a directional interpolation-based approach the M2L matrices do not depend on the directions (see section \ref{subsection_interpolationbasedFMM}). We thus exploit a similar construction algorithm as presented in \cite{engquistying,borm15}: starting from a given regular solid, the set of projected face centers on the unit sphere are chosen as directions at the first high-frequency level. Then, each face is subdivided into $2^{d-1}$ other faces and the process is repeated recursively in order to obtain the direction tree. 
Engquist \& Ying \cite{engquistying} relies on the $d$-cube for the direction generation: according to our tests, the $d$-cube provides indeed the best compromise between the number of directions and the obtained 
accuracy among the platonic solids. 
Since, a $d$-cube has $d!$ faces, there are $d!2^{E(d-1)}$ directions at the $E^{th}$ high-frequency level in \textit{defmm}.

\subsection{Algorithmic design}
\label{ss:design}
\subsubsection{$2^d$-tree construction}
\label{ss:tree_construction}
$2^d$-trees can be built either by setting the maximum number of particles per leaf (denoted as $Ncrit$),
or by setting the maximum possible level (denoted as $MaxDepth$). 
For (highly) non-uniform particle distributions, the $MaxDepth$ strategy leads to (numerous) empty cells,
which should not be stored and processed in practice: this requires advanced data structures 
(see e.g. \cite{naborskorsmeyerleightonwhite94,coulaudfortinroman10,hariharanalurushanker02}). %%, 183]. 
The $MaxDepth$ strategy may also generate   
cells with a greatly varying number of particles
\cite{pavlovandonovkremenliev14}. 
Moreover, a comparison of hierarchical methods in \cite{fortinathanassoulalambert11}
has detailed how the $Ncrit$ strategy better adapts to highly non-uniform distributions than the $MaxDepth$ one. 
Considering the highly non-uniform distributions used in BIE problems, we therefore choose to rely on the $Ncrit$ 
strategy in \textit{defmm}.
This 
differs from  
\textit{dfmm} 
\cite{messner12,messnerINRIA} which relies on the $MaxDepth$ strategy.
 
\label{subsubsection_treeconstructionanddatalocality}
Concretely, our $2^d$-tree construction algorithm is similar to the 
\textit{exafmm}\footnote{State-of-the-art parallel C++ (non-directional) FMM library: \url{https://github.com/exafmm/exafmm}} 
one: the particles are first sorted according to the Morton ordering, 
then each cell is built 
with a pointer on its first particle.  
The particles within a cell 
are thus stored consecutively in memory and followed by  
the particles of the next cell in  
the Morton ordering.  
This ensures that the data is localized for the P2P evaluations. In addition, since the particles of the sons of a given cell are stored consecutively, 
we can directly 
apply P2P operators on non-leaf cells. 
The charges and potentials are stored in two dedicated arrays whose entries correspond to the sorted particles for the same data-locality reason.

We also store all the directional expansions associated to a given cell $c$ in the same array to enhance data locality during the translations of local and multipole expansions: the M2M and L2L evaluations regarding $c$  
are indeed performed alltogether (see section \ref{subsection_blasbasedoperatorsoptimization} for the optimizations of these steps).

\subsubsection{Dual Tree Traversal}
\label{sss:DTT}

\begin{algorithm}[t]
    \caption{\label{alg_dtt}Dual Tree Traversal (DTT) between target cell $t$ and source cell $s$}
    \begin{algorithmic}[1]
        \STATE 
        \COMMENT{$\cc{A}$:  
        strict (resp. directional) MAC in  
        low- (resp. high-) frequency regime}  
        \IF{$\cc{A}(t,s)$} 
            \STATE apply the M2L operator between $t$ and $s$
            \RETURN
        \ELSE
            \IF{$t$ is a leaf \textbf{or} $s$ is a leaf}
                \STATE apply the P2P operator between $t$ and $s$
                \RETURN
            \ENDIF
            \FOR{$t'\in Sons(t)$}
                \FOR{$s'\in Sons(s)$}
                    \STATE DTT($t',s'$)
                \ENDFOR
            \ENDFOR
        \ENDIF
    \end{algorithmic}
\end{algorithm}

Using the $Ncrit$ criterion 
however complexifies the interaction list structure (see lists U, V, W and X in e.g. \cite{yingbiroszorin04}). 
In directional FMMs, the DPSC for well-separate\-ness (see equation (\ref{e:DPSC})) further complicates the construction of the interaction lists (multiple lists for each target cell with varying sizes and shapes depending on the tree level). We therefore adapt here the Dual Tree Traversal (DTT) \cite{denen02,yokota13} to directional FMMs. 
The DTT is a simple and recursive procedure that simultaneously traverses the target and source trees to evaluate \textit{on-the-fly} the M2L and P2P operators by dynamically testing the MAC on pairs of target and source cells. Hence, the interaction lists are traversed and evaluated implicitly but never built. This is another 
difference 
with \textit{dfmm} which relies on a DTT-like algorithm to \textit{explicitly} build the interaction lists during a precomputation step, and then separately processes these lists for each target cell. 

Since FFTs can only efficiently accelerate M2L operations between cells at the same level (because of the relatively small involved circulant embedding), 
our DTT differs from the original one \cite{denen02,yokota13} as follows: 
(i) when the MAC fails, both cells are split (not the largest one); 
(ii) a P2P operation is required as soon as one of the two cells is a leaf (not both).  

We also aim at one single DTT algorithm for both the low- and high-frequency regimes. 
This requires the MAC to depend only on the cell radii and distances, and not on directions. 
Fortunately, this is possible in a directional polynomial interpolation-based FMM, since the M2L matrices do not depend on the directions. 
Our DTT hence relies in the low-frequency regime on the strict MAC (see section \ref{def:strictMAC}), which is simple 
and allows the most far-field interactions for each target cell, 
and on the following new directional MAC in the high-frequency regime (for wavenumber $\kappa$) between cells $t,s$ with radius $w$ at the same tree level:  
\begin{equation}
\label{macdirref}
    \frac{max\{\kappa w^2,2w\}}{dist(t,s)}\leq \eta,
\end{equation}
where $\eta > 0$ ($\eta=1$ in practice). This new directional MAC is similar to another directional MAC \cite{borm15}, but contrary to this latter our MAC does not depend on the directed wedge condition: we do not consider directions in this MAC but only ratios between cell radii and the cell distance. This means that the DTT performs in the high-frequency regime in a similar way than in the low-frequency one, without considering the directional aspects that are entirely hidden in the M2L applications.
We emphasize that the interpolation process using equispaced grids is consistant according to section \ref{section_consistency} on cells complying with these two MACs. 
When this directional MAC is satisfied, one then has to retrieve the best direction 
to select the 
directional multipole expansion (from source $s$) and the directional local expansion (from target $t$) which are relevant for the M2L operation between $t$ and $s$.
The search of the best direction may have a non-negligible cost, but there is one single 
direction associated to each  
M2L matrix. 
We thus precompute these best directions during the M2L matrix precomputation step 
(see  
section \ref{subsubsection_blankpasses}). 

In the end, we obtain the simple DTT algorithm presented in Algorithm \ref{alg_dtt} to differentiate in \textit{defmm} 
the far-field M2L operations from the near-field P2P ones, in both low- and high-frequency regimes.

One may notice that list-based approaches (i.e. without DTT) allow to group or "stack" 
multiple M2L operations (written as matrix-vector products), 
into matrix-matrix products (see e.g. \cite{coulaudfortinroman10,messnerINRIA,malhotrabiros16}). 
This enables indeed to benefit from the higher efficiency of level-3 BLAS routines, especially for uniform distributions.
We refer to this technique as  \textit{vector stacking}. 
Here however, we have M2L operations corresponding to Hadamard products whose grouping cannot lead to more efficient level-3 BLAS operations. 
Such Hadamard products may be converted to matrix-vector products as shown in  \cite{malhotrabiros16}, but this requires extra zeros for non-uniform particle distributions. Considering the highly non-uniform distributions typical of BIE problems, we believe that the potential gain would be too limited. 
We hence do not consider vector stacking for our Hadamard products, and we rather rely on the DTT to efficiently process the BIE non-uniform distributions.  

\subsubsection{Blank passes}
\label{subsubsection_blankpasses}

\begin{algorithm}[t]
    \caption{\label{alg_blankdtt}Blank Dual Tree Traversal (BDTT) between target cell $t$ and source cell $s$}
    \begin{algorithmic}[1]
        \IF{$t$ is in the high-frequency regime} 
        \IF{$\cc{A}(t,s)$}
            \STATE  
            \COMMENT{$\cc{D}(Level(t))$ is the level of $t$ in the direction tree:\,    }
            \STATE compute direction $u:=\displaystyle\mathop{argmin}_{v\in
            \cc{D}(Level(t))}\Big|v-\frac{ctr(t)-ctr(s)}{|ctr(t)-ctr(s)|}\Big|$
            \STATE mark $t$ and $s$ with $u$ \label{alg_blankdtt:mark_u}
            \IF{the M2L matrix $\bb{D}_{s,t}$ corresponding to $t$ and $s$ is not precomputed}
                \STATE precompute $\bb{D}_{s,t}$ 
            \ENDIF
            \STATE mark $\bb{D}_{s,t}$ with $u$
            \RETURN
        \ELSE
            \FOR{$t'\in Sons(t)$}
                \FOR{$s'\in Sons(s)$}
                    \STATE BDTT($t',s'$)
                \ENDFOR
            \ENDFOR
        \ENDIF
        \ENDIF
    \end{algorithmic}
\end{algorithm}

\begin{algorithm}[t]
    \caption{\label{alg_blankdwnpass}Blank Downward Pass (BDP) for cell $c$}
    \begin{algorithmic}[1]
        \IF{$c$ is in the high-frequency regime}
        \FOR{each direction $u$ with which $c$ is marked}
            \FOR{$c'\in Sons(c)$} 
                \STATE mark $c'$ with $Father(u)$
            \ENDFOR
        \ENDFOR
        \FOR{$c'\in Sons(c)$}
            \STATE BDP($c'$)
        \ENDFOR
        \ENDIF
    \end{algorithmic}
\end{algorithm}

On non-uniform distributions 
only a subset of all possible directional expansions will be used in practice due to the DPSC. 
Hence, \textit{defmm} determines the actually needed directional expansions (named \textit{effective} expansions) during a precomputation step and only computes and stores these effective expansions to save memory and computations.  
In this purpose,  
\textit{defmm} performs first a \textit{blank DTT} (see Algorithm \ref{alg_blankdtt}) 
to compute each required M2L matrix and to mark all M2L-interacting cells with the corresponding direction
(see line \ref{alg_blankdtt:mark_u} in Algorithm \ref{alg_blankdtt}). 
Then, thanks to a  
\textit{blank downward pass} (see algorithm \ref{alg_blankdwnpass}), the required directions are propagated down to the leaves of the $2^d$-trees.

\subsection{Exploiting symmetries in the Fourier domain}
\label{ss:FFT_symmetries}

To minimize the number of M2L matrices to be precomputed, in the case of centered expansions
in $2^d$-trees, it has been observed in the literature that 
symmetries can be used
(see for instance \cite{messnerINRIA,engquistying}). The underlying symmetry group is actually
the hyperoctahedral one \cite{theseigor}
(i.e. the octahedral group in 3D, corresponding to the symmetry group of the cube), denoted
by $\kk{D}_d$ in dimension $d$, that can be realized as 
a group of rotation matrices. The use of
symmetries strongly reduces the precomputation cost, especially in the high-frequency regime.
Hence, we want to also exploit these symmetries in \textit{defmm}. To do so, one has to express
these symmetries in the Fourier domain, since the (modified diagonal) M2L matrices are expressed
in this domain.

\subsubsection{Taking symmetries into account}

Let us first detail how, in a given cell $c$ of the (source or target) cluster tree, a permutation
matrix over the interpolation grid $\Xi_c$ can be associated to a symmetry. Take any rotation
$R:\mathbb{R}^d\to \mathbb{R}^d$ that leaves the unit cube centered at $0$ invariant. The set of
rotations satisfying this property forms the so-called hyperoctahedral group.
The translated grid  $\Xi_c-ctr(c)$ is centered at the origin, so
$\Xi_c - ctr(c) = R(\Xi_c-ctr(c))$ i.e. $\Xi_c =  ctr(c) + R(\Xi_c-ctr(c))$.
This transformation, represented in Figure \ref{fig_quotientsetfullset}, 
induces a linear map $\mathbb{R}_c:\mathbb{C}[\Xi_c]\to \mathbb{C}[\Xi_c]$
defined by
\begin{equation}\label{PermutationRepresentation}
  \begin{aligned}
    & \mathbb{R}_c(u)(\bo{p}) := u(ctr(c) + R(\bo{p}-ctr(c)))\\
    & \forall u\in \mathbb{C}[\Xi_c],\;\forall \bo{p}\in \Xi_c.
  \end{aligned}
\end{equation}
This matrix $\mathbb{R}_c$ is a permutation of interpolation nodes, and the correspondence
$R\leftrightarrow \mathbb{R}_c$ implements a permutation representation of the hyperoctahedral
group. This is the group representation that we propose to exploit 
to minimize the number of precomputed M2L matrices. 

\begin{figure}[t]
  \centering
  \includegraphics[width=0.3\linewidth]{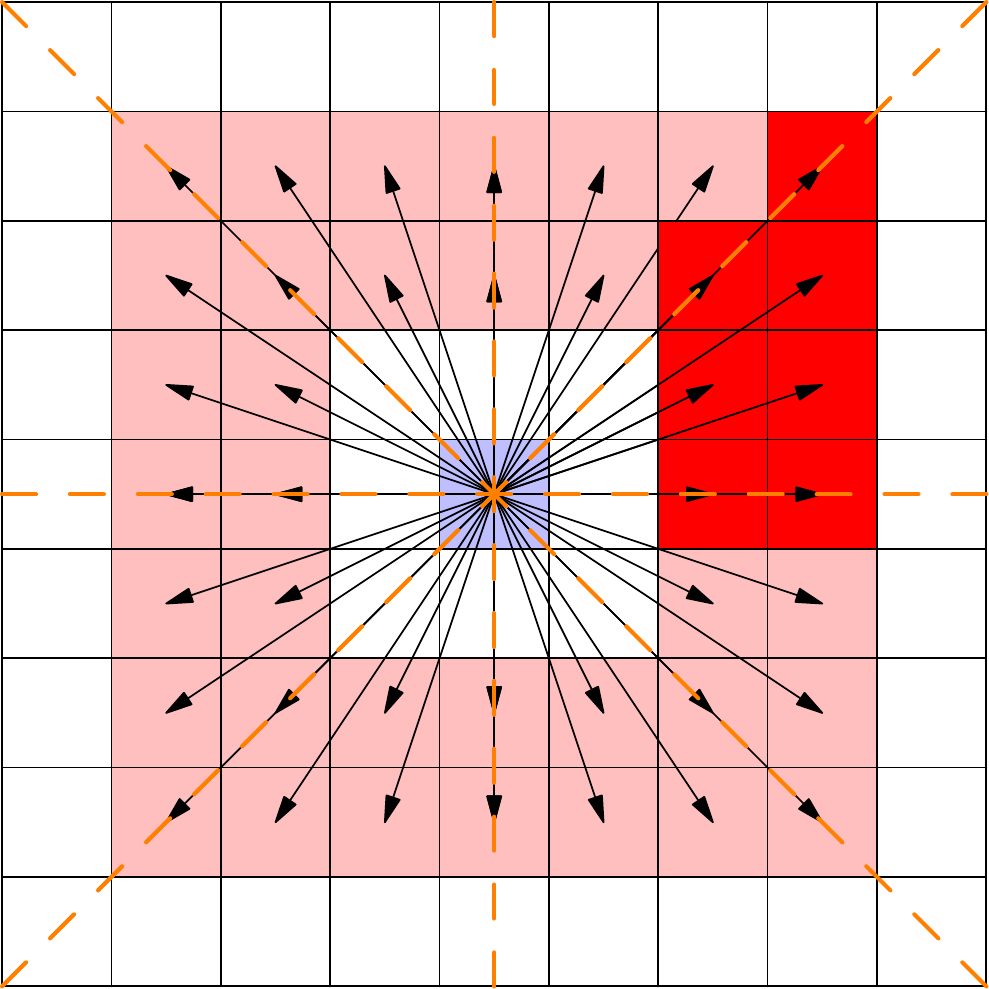}
  \hspace{0.2cm}
  \includegraphics[width=0.3\linewidth]{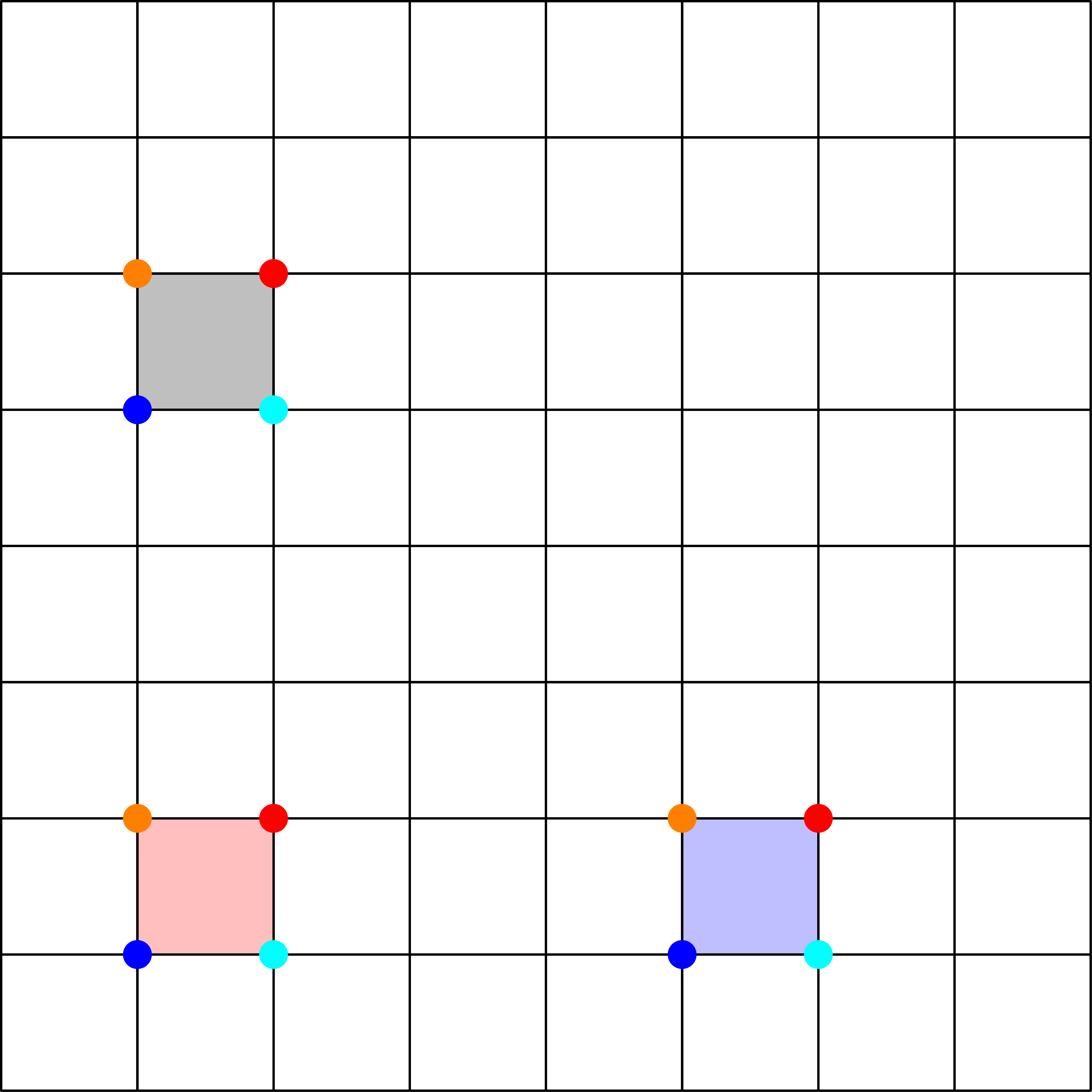}
  \hspace{0.2cm}
  \includegraphics[width=0.3\linewidth]{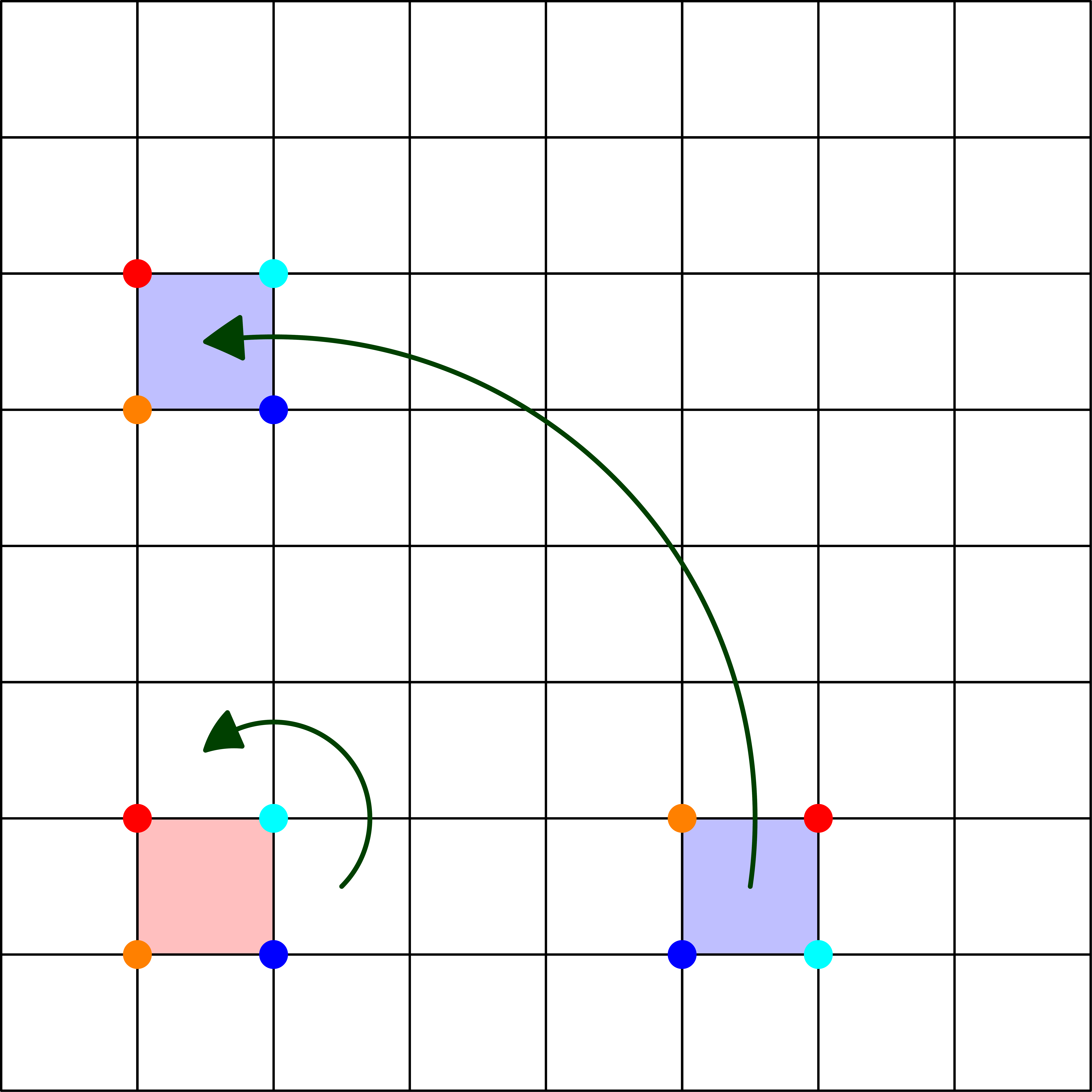}
  \caption{\label{fig_quotientsetfullset}\underline{Left:}
    Source cell $s$ in the low-frequency regime (blue) with possible target
    cells $t$ (red) such that $s$ is in the interaction list of $t$. The
    corresponding M2L matrices for cells $t$ in pale red can be all deduced (for
    instance) from permutations of the dark red ones. Symmetry axes of the square
    ($2$-cube) are represented in orange. \underline{Middle and right:} M2L matrix
    between the red and blue well-separated cells (middle) as a
    rotation of the M2L matrix between the red and grey cells (right), 
    that permutes the relative positions of the interpolation nodes (colored dots).
    Since the entries of the M2L matrices 
    correspond to interpolation nodes, permutations matrices represent this
    process. The underlying symmetry of $\kk{D}_2$ is the reflection with regard to
    the line $x=y$.}
\end{figure}

\quad\\ 
Denote $\mathbb{R}_\star:\mathbb{C}[\Xi_\star]\to \mathbb{C}[\Xi_\star]$
the permutation induced on $\Xi_\star$ by means of the correspondence described above.
When the same rotation transformation is applied to both source and target
cells $s$ and $t$, the transformed M2L matrix writes
\begin{equation}
  \begin{aligned}
    \mathbb{R}_t^\top G(\Xi_t,\Xi_s) \mathbb{R}_s
    & = \mathbb{R}_\star^\top \mathcal{G}_{s,t}(\Xi_\star,\Xi_\star) \mathbb{R}_\star\\
    & = \mathbb{R}_\star^\top\chi^\top \mathbb{F}^* \mathbb{D}_{s,t}\mathbb{F}\chi\mathbb{R}_\star.
  \end{aligned}
\end{equation}
We are going to show that the product above actually writes
$\chi^\top \mathbb{F}^* \tilde{\mathbb{D}}_{s,t}\mathbb{F}\chi $ where $\tilde{\mathbb{D}}_{s,t}$
is a diagonal matrix mapping $\mathbb{C}[\widehat{\Xi}_\star]\to \mathbb{C}[\widehat{\Xi}_\star]$
with coefficients obtained by permuting the diagonal of $\mathbb{D}_{s,t}$. 

In concrete computational terms, this means that, whenever two M2L interactions only differ
by a rotation, 
one M2L matrix is deduced from the other by a simple permutation of the
diagonal entries of the factor term  $\mathbb{D}_{s,t}$, which reduces storage and precomputation
cost. 

\quad\\
Let $c_\star = (1/2,\dots,1/2) = ctr(\Xi_\star)\in \mathbb{R}^d$ refer to the center of the normalized
interpolation grid. Denote $\Xi_0 = \Xi_\star-ctr(\Xi_\star)$ the normalized interpolation grid
translated so as to be centered at $0$. Take two vectors $u,v\in \mathbb{C}[\Xi_\star]$. Then
we have
\begin{equation}
  \begin{aligned}
    & v^\top\mathbb{R}_\star^\top \mathcal{G}_{s,t}(\Xi_\star,\Xi_\star) \mathbb{R}_\star u\\
    & = \sum_{\bo{p}\in \Xi_\star} \sum_{\bo{q}\in \Xi_\star}\mathcal{G}_{s,t}(\bo{p}-\bo{q})
    u(c_\star + R(\bo{p}-c_\star))v(c_\star + R(\bo{q}-c_\star))\\
    & = \sum_{\bo{x}\in \Xi_0} \sum_{\bo{y}\in \Xi_0}\mathcal{G}_{s,t}(\bo{x}-\bo{y})
    u(c_\star + R(\bo{x}))v(c_\star + R(\bo{y}))\\
    & = \sum_{\bo{x}'\in \Xi_\star} \sum_{\bo{y}'\in \Xi_\star}\mathcal{G}_{s,t}(R^*(\bo{p}-\bo{q}))
    u(\bo{p})v(\bo{q}).\\    
  \end{aligned}
\end{equation}
Since $u,v$ are arbitrarily chosen in $\mathbb{C}[\Xi_\star]$, this can be rewritten in
condensed form by $\mathbb{R}_\star^\top \mathcal{G}_{s,t}(\Xi_\star,\Xi_\star) \mathbb{R}_\star =
(\mathcal{G}_{s,t}\circ R^*)(\Xi_\star,\Xi_\star)$. Now there only remains to return to the
calculus presented in section \ref{subsection_equispacedgrids} that shows
\begin{equation}
\label{eq:permuteM2L}
  \begin{aligned}
    & \mathbb{R}_\star^\top \mathcal{G}_{s,t}(\Xi_\star,\Xi_\star) \mathbb{R}_\star =
    \chi^\top\mathbb{F}^*\,\mathbb{D}[\mathcal{G}_{s,t}\circ R^*]\,\mathbb{F}\chi\\
    & \text{with}\quad \mathbb{D}[\mathcal{G}_{s,t}\circ R^*] =
    \mathrm{diag}(\,\mathbb{F}(\mathcal{G}_{s,t}\circ R^*)\,).
  \end{aligned}
\end{equation}
To summarize, the only difference between $\mathbb{R}_\star^\top \mathcal{G}_{s,t}(\Xi_\star,\Xi_\star) \mathbb{R}_\star$
and $\mathcal{G}_{s,t}(\Xi_\star,\Xi_\star)$ lies in the coefficients of the Fourier symbol in the central diagonal
term of their factorized form. Let us examine how the rotation $R^*$ acts on the Fourier symbol of
$\mathcal{G}_{s,t}$. According to \eqref{DFTformula}, and since $R(\Xi_\sharp) = \Xi_\sharp$,
we have
\begin{equation}
  \begin{aligned}
    \mathbb{F}(\mathcal{G}_{s,t}\circ R^*)(\boldsymbol{\xi})
    & = \frac{1}{T^{d/2}}\sum_{\bo{x}\in \Xi_\sharp}
    \mathcal{G}_{s,t}(R^*\bo{x})
    \exp(-2\imath \pi\, \boldsymbol{\xi}\cdot \bo{x})\\
    & = \frac{1}{T^{d/2}}\sum_{\bo{y}\in \Xi_\sharp}
    \mathcal{G}_{s,t}(\bo{y})
    \exp(-2\imath \pi\, R^*(\boldsymbol{\xi})\cdot \bo{y}) =
        \mathbb{F}(\mathcal{G}_{s,t})( R^*\boldsymbol{\xi})
  \end{aligned}
\end{equation}
which again summarizes as $\mathbb{F}(\mathcal{G}_{s,t}\circ R^*) = \mathbb{F}(\mathcal{G}_{s,t})\circ R^*$.
In other words, the permutation on the grid $\widehat{\Xi}_\sharp$ that should be applied on the symbol of
the periodized M2L operator is the permutation associated to the inverse rotation $R^* = R^{-1}$.

\subsubsection{Symmetries in practice}

Usually, 
the multipole and local expansions are permuted respectively before and after the evaluation of a M2L operator 
(see e.g. \cite{messnerINRIA,malhotrabiros16}). Here we can reduce this number of applied permutations to only one thanks to the diagonal form of the M2L matrices in the Fourier domain in \textit{defmm} (see equation \eqref{eq:permuteM2L}). Indeed, the permutations can all be applied  
to these diagonal M2L matrices, resulting in a permutation of their diagonal entries.
          This reduces the number of permutations from two to only one.
          
          In practice, the permutations induce memory indirections when performing the Hadamard products that may prevent the compiler auto-vectorization. We hence rely on an OpenMP\footnote{See: \url{https://www.openmp.org/}} directive to enforce the compiler vectorization 
          (see \cite{theseigor} for details). 
          
\section{Optimizations}
\label{section_optimizations}
We here present algorithmic and programming optimizations regarding high performance computing on one CPU core.

\subsection{Processing the FFTs}
\label{subsection_processingtheFFTs}

   \begin{figure}[t]
        \centering
        \includegraphics[width=0.7\linewidth]{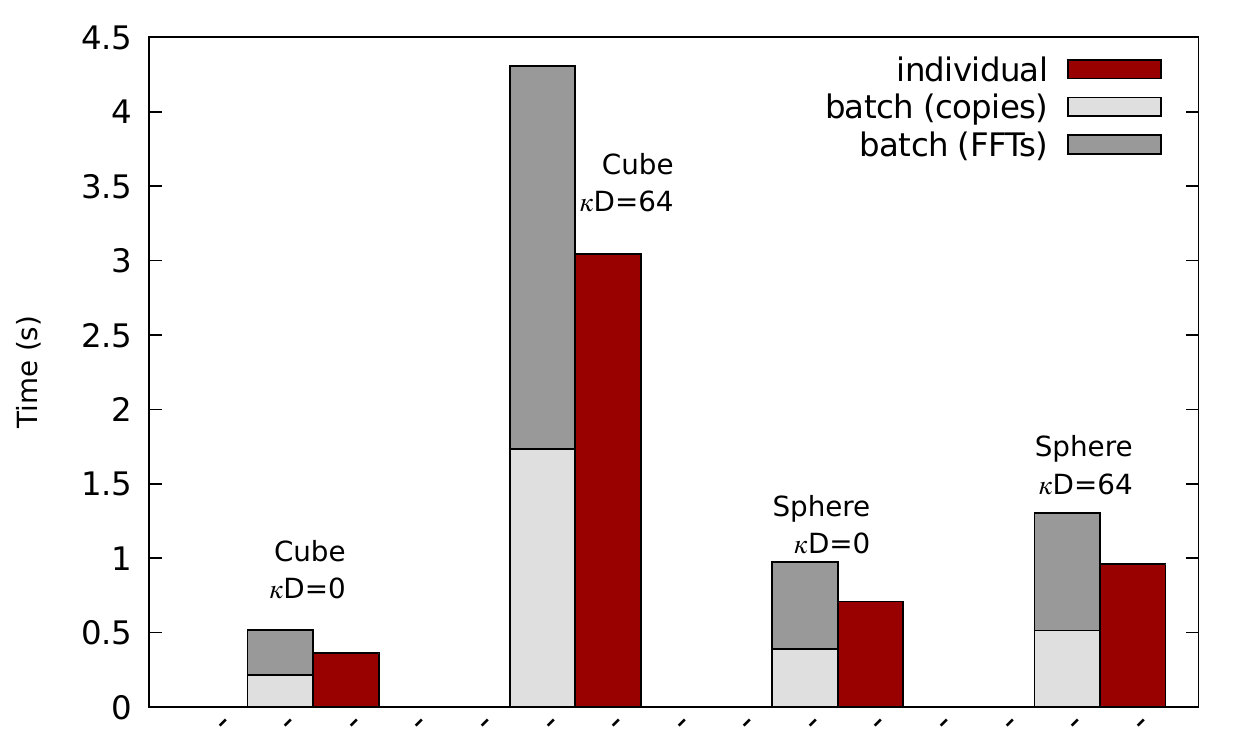}
        \caption{\label{fig_dufmm_comparison_intensivebatch}Timings of all 
        FFT applications with 1D interpolation order $4$, using the batch and individual methods. 
        Tests performed 
        with $10^7$ particles (see section 
        \ref{section_numerical}
        for details).}
    \end{figure}

We rely on the state-of-the-art FFTW library 
\cite{FFTWref} to efficiently perform all at once our numerous FFTs.   
Our 3D FFTs (considering $d=3$) are however small: $(2 L -1) ^d $ elements, with usually $L \leq 7$.
This makes the FFTW "plan" creation, required to decide the most efficient FFT implementation, 
more time consuming than its execution (i.e. the actual FFT computation). 
Moreover, 
zero-padding is required for the circulant embedding (see section \ref{ss:FFT_symmetries}) of the expansions stored in the Fourier domain: this implies expansion copies before each FFT, as well as after each reverse FFT.  

One could first use the "batch" FFTW feature\footnote{See: \url{http://www.fftw.org/fftw3\_doc/Advanced-Complex-DFTs.html}}
to efficiently process our numerous small FFTs. This performs all FFTs with one single FFTW plan and can improve performance 
compared to multiple individual FFT calls. All expansion copies hence have to be performed all together before and after the batch FFTW call. 
We refer to this method as the \textit{batch} one. 

Since all our expansions have the same size and the same memory alignment, we can also rely on one single FFTW plan for all our FFTs, and perform 
individually each FFT along with its expansion copy. 
This \textit{individual} method benefits from the cache memory for our small-sized expansions, and hence avoids to load data twice from main memory (for FFTs and for copies, which are both memory-bound operations) as done in the batch method. This is shown in figure \ref{fig_dufmm_comparison_intensivebatch}, where
the individual method always outperforms (up to 30\%) the batch one. 
We thus use the individual method in \textit{defmm}.

\subsection{BLAS-based upward and downward pass operators}
\label{subsection_blasbasedoperatorsoptimization}
Because of the directional aspects of \textit{defmm}, 
the cost of the upward and downward passes is linearithmic for surface meshes (see \cite{engquistying}). As opposed to FMMs  
for non-oscillatory kernels, these steps have a significant cost in practice,
which justifies their careful optimization. We hence first adapt an optimization suggested in \cite{agullobramascoulauddarvemessnertoru12}
to equispaced grids and to oscillatory kernels (see section \ref{dufmm_M2M_tensorized_algo}), and then improve it 
for directional FMMs  
(see sections \ref{section_dufmm_direction_stacking} and \ref{section_dufmm_realcomplex_stacking}). 
We validate our optimizations in section \ref{ss:M2M_L2L_perf}. 

      \subsubsection{Tensorized M2M and L2L operators}
      \label{dufmm_M2M_tensorized_algo}
      We here detail  
      a fast evaluation scheme  suggested in \cite{agullobramascoulauddarvemessnertoru12} for the M2M and L2L operators on interpolation grids using tensorized Chebyshev rules in the low-frequency regime. 
      We extend this scheme to  
      equispaced grids which are 
      also tensorized grids. The bijection in definition \ref{definition_dufmm_I} induces a node indexing 
      in the interpolation grid $\bb{G}$ allowing to exploit this tensorized structure.
        \begin{definition}
            \label{definition_dufmm_I}
            Let $L\in \bb{N}$. $\kk{I}$ denotes the bijection from $[\![0,L^d-1]\!]$ to $[\![0,L-1]\!]^d$ such that $\kk{I}^{-1}(\bo{I}) := \sum_{k=1}^dI_kL^{k-1}$, $\forall\hspace{0.1cm}\bo{I} := \left( \kk{I}_1,..., \kk{I}_d\right)$.
         \end{definition}
      
      Thanks to this tensorized structure, the matrix representations of the M2M and L2L operators are tensorized matrices. 
      The L2L case being obtained by transposition, we focus on the M2M case. 
      Let $\bb{M}\in\bb{R}^{L^d\times L^d}$ be the matrix representation of a M2M operation, 
      there exists therefore $M^{(p)}\in \bb{R}^{L\times L}$, $p\in [\![1,d]\!]$, such that $\bb{M} = \displaystyle\mathop{\otimes}_{p=1}^dM^{(p)}$. Using definition \ref{definition_dufmm_I}, the following holds
      \begin{equation*}
          \left(\mathop{\otimes}_{k=1}^dM^{(k)}\right)_{i,j} = \prod_{k=1}^dM^{(k)}_{\kk{I}(i)_k,\kk{I}(j)_k}.
      \end{equation*}

      For any $\bo{v}\in \bb{C}^{L^d}$, we have
      \begin{equation*}
        \begin{aligned}
          \left(\bb{M}\bo{v}\right)_{i} &= \left(\left(\mathop{\otimes}_{k=1}^dM^{(k)}\right)\bo{v}\right)_{i} = \sum_{j=0}^{L^d-1}\inpar{\mathop{\prod}_{p=1}^dM^{(p)}_{\kk{I}(i)_p,\kk{I}(j)_p}}v_{j}\\
          &= \sum_{\bo{J}\in [\![0,L-1]\!]^d}\mathop{\prod}_{p=1}^dM^{(p)}_{\kk{I}(i)_p,\bo{J}_p}v_{\kk{I}^{-1}(\bo{J})} = \sum_{\bo{J}\in [\![0,L-1]\!]^d}\mathop{\prod}_{p=1}^{d-1}M^{(p)}_{\kk{I}(i)_p,\bo{J}_p}\left(M^{(d)}_{\kk{I}(i)_d,\bo{J}_d}v_{\kk{I}^{-1}(\bo{J})}\right)\\
          &= \sum_{\substack{\bo{J}\in [\![0,L-1]\!]^d\\\bo{J}_d=0}}\mathop{\prod}_{p=1}^{d-1}M^{(p)}_{\kk{I}(i)_p,\bo{J}_p}\left(\sum_{q=0}^{L-1}M^{(d)}_{\kk{I}(i)_d,q}v_{\kk{I}^{-1}(\bo{J}+q\bo{e}_d)}\right)
        \end{aligned}
      \end{equation*}
      where the last sum over $q$ 
      matches  
      a matrix-vector product of size $L\times L$. For  
      $i$ varying, this matrix-vector product is performed on $L^{d-1}$ different restrictions of $\bo{v}$, 
      involving each time the same matrix. Hence, there exists a permutation $P\in\bb{R}^{L^d\times L^d}$ such that
      \begin{equation*}
        \begin{aligned}
          \left(\bb{M}\bo{v}\right)_{i} &= \sum_{\substack{\bo{J}\in [\![0,L-1]\!]^d\\\bo{J}_d=0}}\mathop{\prod}_{p=1}^{d-1}M^{(p)}_{\kk{I}(i)_p,\bo{J}_p}\left( diag(M^{(d)})P\bo{v}\right)_{\kk{I}(i)_d}
        \end{aligned}
      \end{equation*}
      where $diag(M^{(d)})$ is a block-diagonal matrix with all diagonal bocks equal to $M^{(d)}$. This process can be repeated $d$ times, leading to an overall complexity of $\cc{O}\left(dL^{d+1}\right)$ since the permutations are applied in $\cc{O}(L^d)$ operations. This compares
      favorably with 
      the $\cc{O}\left(L^{2d}\right)$ complexity of a naive approach.  
      Since the same matrix is used for multiple vectors at 
      each of the $d$ iterations, matrix-vector products can be stacked into matrix-matrix products 
      to benefit from the level-3 BLAS higher efficiency 
      \cite{agullobramascoulauddarvemessnertoru12}. 
      We will refer to this version as the \textit{tensorized} (or $t$) method.  
      
      The extension of the tensorized method to  
      oscillatory kernels  
      is obtained by noting that the directional M2M matrix $\bb{M}(u)$ with direction $u$ can be written
      \begin{equation}
        \label{eq_dufmm_dlefturight}
        \bb{M}\left( u\right) = D_{0}(u)\left( \mathop{\otimes}_{p=1}^dM^{(p)}\right)D_{1}(u)
      \end{equation}
      with two diagonal matrices $D_{0}(u)$ and $D_{1}(u)$ composed of complex exponential evaluations (see section \ref{subsection_interpolationbasedFMM}).

      \subsubsection{Directional stacking}
      \label{section_dufmm_direction_stacking}
      The tensorized me\-thod can be further optimized in the high-frequency regime. 
      Starting from equation \ref{eq_dufmm_dlefturight}, if we consider for a given cell two directions $u$ and $u'$ with corresponding directional multipole expansions $\bo{v}(u)$ and $\bo{v}(u')$, denoting by $\odot$ the Hadamard product, the "stacking"
      \begin{equation*}
        \begin{aligned}
          \begin{bmatrix}
            \bb{M}(u)\bo{v}(u) &
            \bb{M}(u')\bo{v}(u')
          \end{bmatrix}
          \end{aligned}
          \end{equation*}
          can be expressed as
          \begin{equation*}
              \begin{aligned}
         \begin{bmatrix}D_{0}(u) & D_{0}(u')\end{bmatrix}
          \odot \inpar{\left(\displaystyle\mathop{\otimes}_{p=1}^dM^{(p)}\right) \cdot \inpar{
          \begin{bmatrix}D_{1}(u) & D_{1}(u')\end{bmatrix}\odot \begin{bmatrix}\bo{v}(u) & \bo{v}(u')\end{bmatrix}}}.
        \end{aligned}
      \end{equation*}
      $\mathop{\otimes}_{p=1}^dM^{(p)}$ here applies to a matrix instead of a vector, which allows to further benefit from the level-3 BLAS  efficiency. 
      Contrary to the tensorized method where $L^{d-1}$ vectors can be stacked, 
      the number of vectors that can be stacked is now $L^{d-1}$ multiplied by the number of effective 
      directional multipole expansions in each cell. 
      This method is referred to as the \textit{tensorized+stacking} method (denoted $t+s$). 

      \subsubsection{Benefiting from real matrix  
      products}
      \label{section_dufmm_realcomplex_stacking}
      One may observe that the matrices $M^{(p)}$ are real since they are composed of evaluations of Lagrange polynomials, but are applied to complex vectors (at least in a directional method). For any $\bo{v}\in \bb{C}^{L^d}$, we thus have
      \begin{equation*}
	\begin{aligned}
	  \left( \mathop{\otimes}_{p=1}^dM^{(p)}\right)\bo{v}  &= \left( \mathop{\otimes}_{p=1}^dM^{(p)}\right)\kk{Re}\{\bo{v}\} + i \left( \mathop{\otimes}_{p=1}^dM^{(p)}\right)\kk{Im}\{\bo{v}\}.
	\end{aligned}
      \end{equation*}
      By deinterleaving the real and imaginary parts in the expansion vectors, one can obtain 
      stacked matrices of real elements.
      Hence, by discarding the imaginary part of the (real) M2M %/ L2L 
      matrices, we can halve the number of required 
      arithmetic operations. 
      This method is referred to as the \textit{tensorized+stacking+real} method (denoted $t+s+r$).
      Since there is only one multipole and one local expansion per cell in the low-frequency regime, no stacking of the directional expansions is performed and the $t+s+r$ method is reduced to $t+r$ in this frequency regime. 
      
      \subsubsection{Performance results}
      \label{ss:M2M_L2L_perf}
      
    \begin{figure}[t]
        \centering
        \begin{subfigure}{.46\textwidth}
            \includegraphics[trim={1cm 0 1cm 1cm},width=\linewidth]{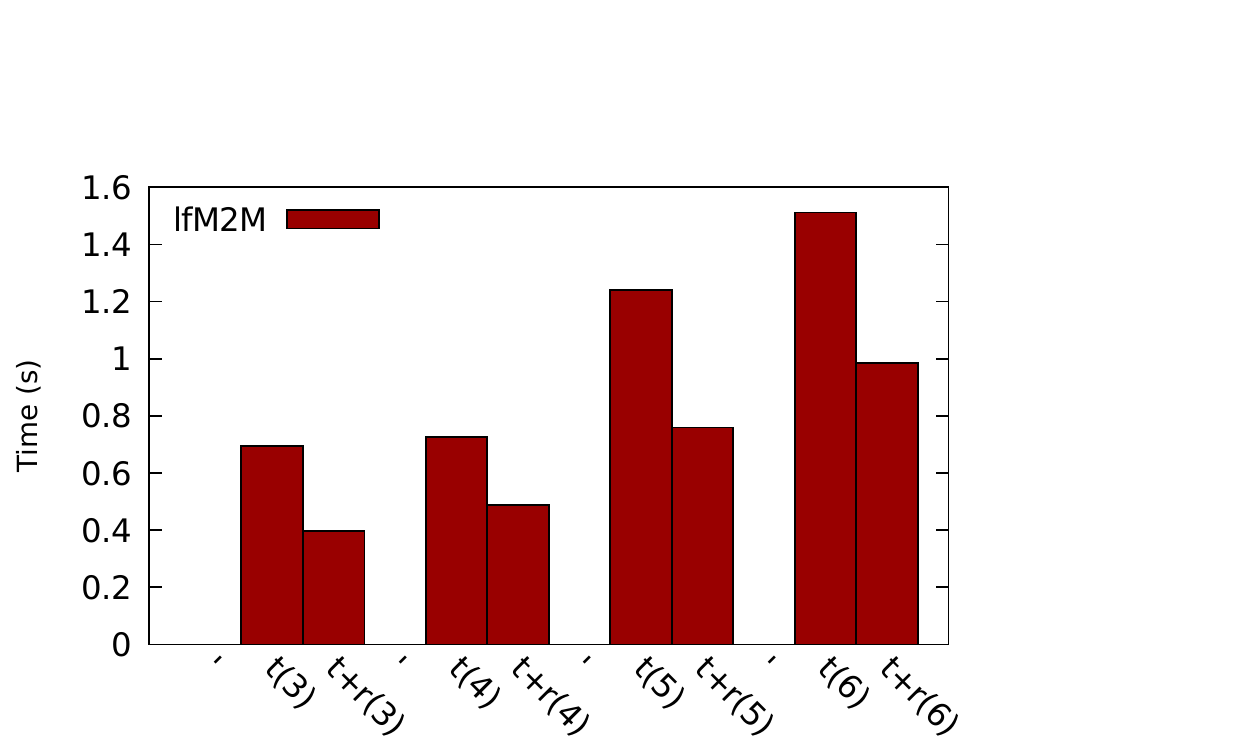}  
            \caption{Sphere $\kappa D = 0$}
            \label{fig:M2Msph0}
        \end{subfigure}
        \begin{subfigure}{.46\textwidth}
            \centering
            \includegraphics[trim={1cm 0cm 1cm 1cm},width=\linewidth]{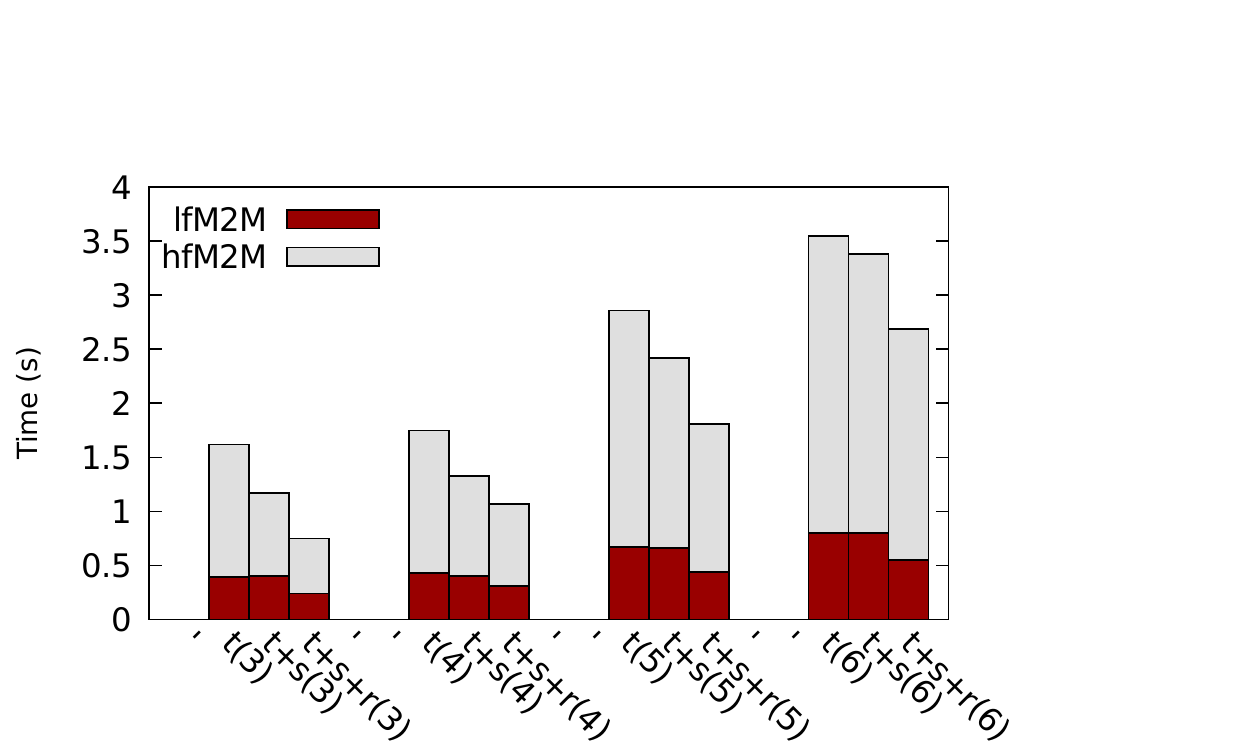}  
            \caption{Sphere $\kappa D = 64$}
        \label{fig:M2Msph64}
        \end{subfigure}
        
        \begin{subfigure}{.46\textwidth}
            \centering
            \includegraphics[trim={1cm 0 1cm 1cm},width=\linewidth]{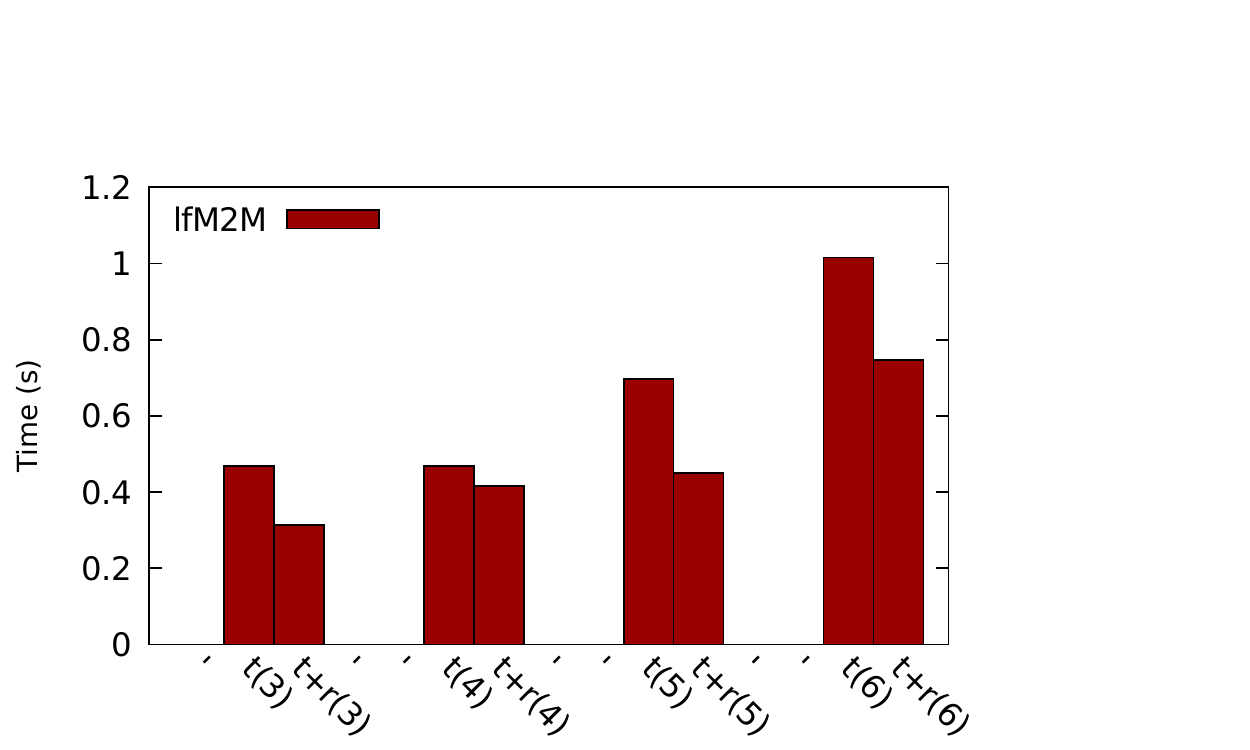}  
            \caption{Cube $\kappa D = 0$}
            \label{fig:M2Mcb0}
        \end{subfigure}
        \begin{subfigure}{.46\textwidth}
            \centering
            \includegraphics[trim={1cm 0 1cm 1cm},width=\linewidth]{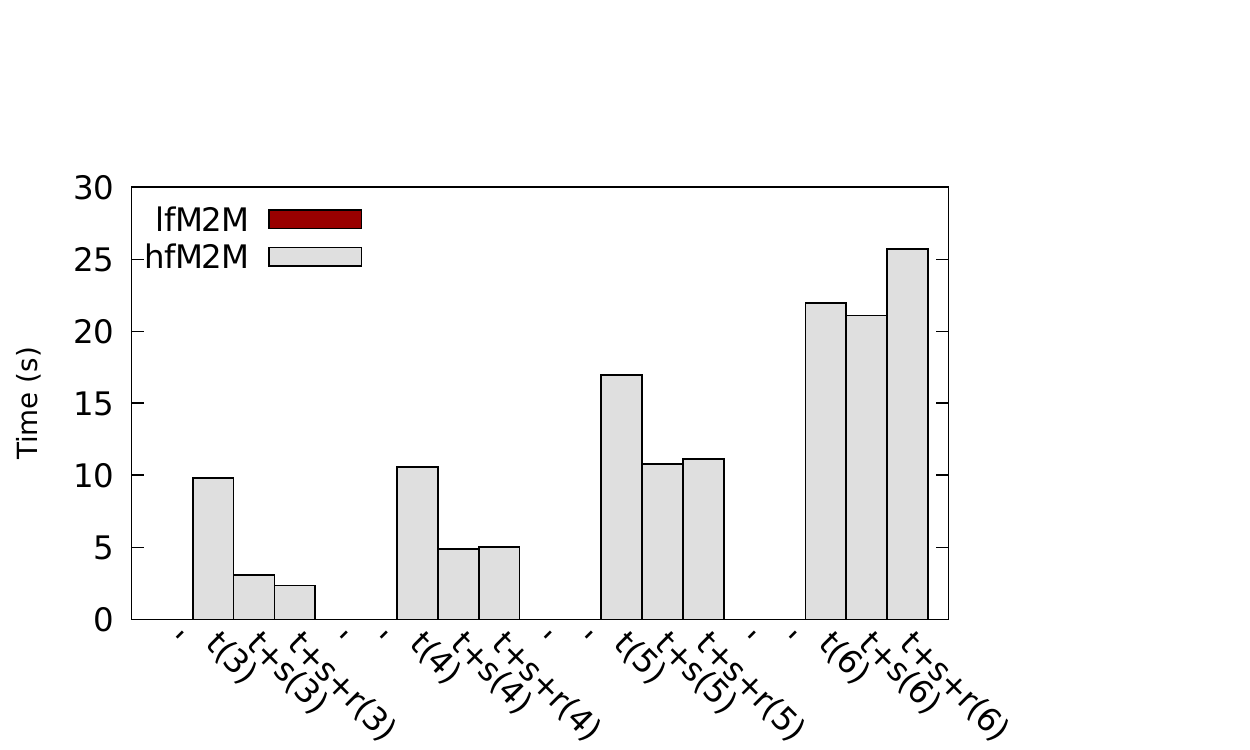}  
            \caption{Cube $\kappa D = 64$}
            \label{fig:M2Mcb64}
        \end{subfigure}
        \caption{\label{fig_dufmm_res_M2M_variants0} Timings of all 
        M2M evaluations for the tensorized ($t$), tensorized+stacking ($t+s$) and tensorized+stacking+real ($t+s+r$) methods, 
        for different distributions with $10^7$ particles and for different frequency regimes. $\kappa D$ refers to the length of the particle distribution multiplied by the wavenumber $\kappa$.
        1D interpolation orders are indicated inside parentheses. lfM2M (respectively hfM2M) refers to the low- (resp. high-) frequency M2M operations.  
        See section \ref{subsubsection_testcasesandcompilers} for a detailed test case description.}
      \end{figure}
      
      In figure \ref{fig_dufmm_res_M2M_variants0} are detailed the  
      M2M timings for these three optimizations.  
      Regarding the sphere with $\kappa D = 64$ (see figure \ref{fig:M2Msph64}),   
      the new $t+s$ method fastens the M2M evaluations of the $t$ method in the high-frequency regime, while our $t+s+r$ method 
      further reduces the M2M evaluation times of this $t+s$ method in the two frequency regimes. 
      In the end, our $t+s+r$ method outperforms the original $t$ method of \cite{agullobramascoulauddarvemessnertoru12} by a factor up to $2.16 \times$.

      This is also valid for the cube with $\kappa D = 64$ (see figure \ref{fig:M2Mcb64}), 
      when considering low interpolation orders: 
      for $L=3$ the $t+s+r$ method outperforms the $t$ one by a factor of $\approx 4 \times$. 
      For larger interpolation orders, the $t+s+r$ method is however less efficient than the $t+s$ one, and even leads to a performance loss for $L=6$. 
      This is due to our small matrix sizes (one matrix being of size $L \times L$) for which the $t+s$ method already provides a high enough number of stacked 
      vectors with the cube distribution: there is thus no benefit with a higher number of stacked vectors (as provided by the $t+s+r$ method). Moreover, the deinterleaving cost is not here offset by the lower number of operations induced by the real matrix products. 
      
      In the low-frequency regime (see figures 
      \ref{fig:M2Msph0},\ref{fig:M2Msph64},\ref{fig:M2Mcb0}), 
      the deinterleaving still allows faster M2M evaluations than the original $t$ method 
      with performance gains up to 4x. 
      Since we target surface particle distributions for BIE problems, we choose to rely on the $t+s+r$ method in \textit{defmm}
      for M2M evaluations, as well as for L2L ones (for which the same performance gains have been obtained).

\subsection{Vectorization for direct computation}
\label{subsection_vectorizationfordirectcomputation}

The direct computation involved in the P2P operator consists in two nested loops over the (possibly distinct) target and source particles. 
As usually performed in FMM codes, we aim at vectorizing this computation to benefit from the AVX2 or AVX-512 units.
We target here the vectorization of the outer loop (on target particles), which leads to fewer reduction operations than with an inner loop vectorization. 
Based on the data locality among particles described in section \ref{subsubsection_treeconstructionanddatalocality}, our P2P code is close to 
the one presented in  
\cite{yokota18}.
As in \cite{yokota18}, we load separately the real and imaginary parts of the complex numbers and we perform only real (not complex) mathematical 
operations: this is required for our compiler (Intel C++ compiler) to fully vectorize the code. 
There are nevertheless some differences between our code and the one presented in \cite{yokota18}:
\begin{itemize}
    \item 
    Since the inputs of our code are the particle point clouds, we have to be able to numerically handle the practical case of interactions between equally located particles on which the kernel function cannot be evaluated.
    These singularities (i.e. the interactions between a particle and itself) are resolved directly: if the distance between the particles is too small, we set one operand to $0.0$ and we continue the interaction computation, which hence results in $0.0$.
    This minimal test may be processed with masks in the vectorized code, hence leading to a minimal performance penalty. 
    \item In order to vectorize our outer loop, an OpenMP compiler directive is required 
    for the Intel C++ compiler (which targets the inner loop otherwise).
    It can be noticed that the GNU C++ compiler currently fails to vectorize our code (with or without the OpenMP directive). 
    \item Since the particle positions are not updated from one FMM application to the next in BIE problems, 
    we can rely on  
    a structure-of-array data layout for the potentials and the charges, which favors efficient vector memory accesses. 
\end{itemize}

In the end, using 64-bit floating-point numbers on 512-bit AVX-512 units (hence a maximum speedup of $8 \times$), 
our vectorized P2P operator offers performance gains up to $7.6 \times$ over the original scalar implementation (see \cite{theseigor}
for details).

\section{Numerical results}
\label{section_numerical}
We now provide numerical results illustrating the convergence and the performance of \textit{defmm}. All codes run sequentially on a Intel Xeon Gold 6152 CPU with AVX-512 units and 384 GB of RAM.
\subsection{Relative error}
We first check the 
overall accuracy of \textit{defmm} on a (volumic) uniform cube test case composed of $125000$ particles. 
We compute the error of a \textit{defmm} approximation $\tilde{p}$ of $p$ (see equation \eqref{fmmsum}) for a norm $||\cdot ||$ as $\frac{||p-\tilde{p}||}{||p||}$. The charges are randomly chosen.

As shown in figure \ref{fig_convcube}, the convergence  follows the estimate derived from the interpolation process on equispaced grids
in section \ref{subsection_dufmm_proofofthemaintheorem}, i.e. a geometric convergence in the 1D interpolation order. 
However, the same interpolation order leads to lower accuracies in the high-frequency regime than in the low-frequency one.
We believe that this is a consequence of the small initial set of directions 
chosen 
for performance reasons in the \textit{defmm} implementation (see section \ref{subsubsection_directiongeneration}). Similar results were obtained on a surface sphere distribution. 

\begin{figure}[t]
    \centering
    \includegraphics[width=0.8\linewidth]{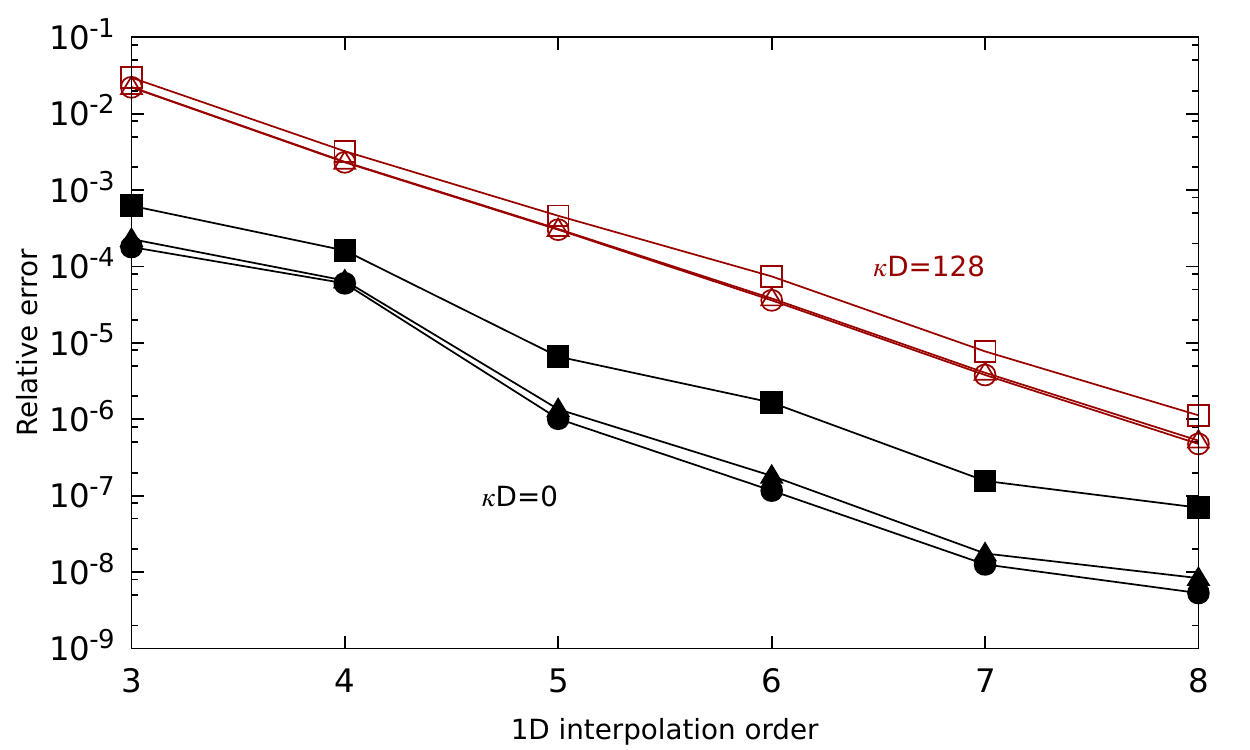}
    \caption{\label{fig_convcube}Relative error of the FMM  
    over a 3D uniform cube with $1.25\times 10^{5}$ particles.
    Maximum norm: squares; $l^1$ norm: circles; $l^2$ norm: triangles.
    }
\end{figure}

\subsection{Performance comparison with \textit{dfmm}}
\label{subsubsection_testcasesandcompilers}
For this performance comparison, we consider the following  
test cases, all with $10^7$ particles ($\pm 1\%$).
\begin{itemize}
    \item The \textit{uniform cube} test case (volumic distribution). This classical FMM test case is not typical for BIE problems, but is still interesting due to its high computation cost 
    for directional FMMs  
    (the number of effective expansions being maximum). 
    \item The classical \textit{sphere} test case where 
    the particles are quasi-uniformly 
    scattered on the surface of the unit sphere. This is also considered as a difficult case, with 
    high memory requirements and high computation costs, 
    for the directional methods on boundary particle distributions \cite{engquistying}. 
    \item The \textit{refined cube}, where the particles are highly non-uniformly sampled along the surface of the unit cube with a higher concentration around the edges and corners. This mimics distributions obtained with mesh refinement.
    \item The \textit{ellipse}, which is an elongated distribution favorable to 
    directional FMMs 
    due to the reduced number of directions in the high-frequency regime. Our ellipse distribution presents a higher concentration of particles on the two poles.
\end{itemize}

\textit{defmm} is compiled using the Intel C++ %\textit{icpc} 
compiler.  
All the BLAS calls use the Intel Math Kernel Library (MKL) %%\cred{version 2019.3.199} 
and we rely on the FFTW3 library for the FFTs.

Regarding \textit{dfmm} (only \texttt{g++} supported), we consider its two best variants \cite{messnerINRIA}, 
both based on low-rank compressions and on interpolations using tensorized Chebyshev grids. 
The first variant (\textit{IAblk} -- Individual Approximation with symmetries and BLocKing) relies on vector stacking and symmetries 
to improve performance with level-3 BLAS routines.
The second one (\textit{SArcmp} -- Single Approximation with ReCoMPression) uses global approximations of the M2L matrices and recompressions to reduce the numerical ranks of these matrices. 
According to \cite{messnerINRIA}, this can lead to faster M2L evaluations than with \textit{IAblk}, but at the cost of longer precomputation times.

Following the methodology used in \cite{messner12}, we fix the threshold in the low-rank approximations of \textit{dfmm}
to $10^{-L}$.
We checked that 
the accuracy of the Chebyshev interpolation in \textit{dfmm} is then similar  
to the accuracy of the equispaced interpolation in \textit{defmm}. 
For each test, the $MaxDepth$ (\textit{dfmm}) and the $Ncrit$ (\textit{defmm}) values (see section \ref{ss:tree_construction}) are tuned to minimize the FMM application time. 
Due to memory constraints, we were not able to run \textit{dfmm} (with \textit{IAblk} or \textit{SArcmp}) on the refined cube for $L=6$. 
The comparison 
results are given in figure \ref{fig_compdefmmdfmm}. 

\begin{figure}[t]
    \begin{subfigure}{.45\textwidth}
  \centering
  \includegraphics[trim={1cm 0 1cm 1cm},width=\linewidth]{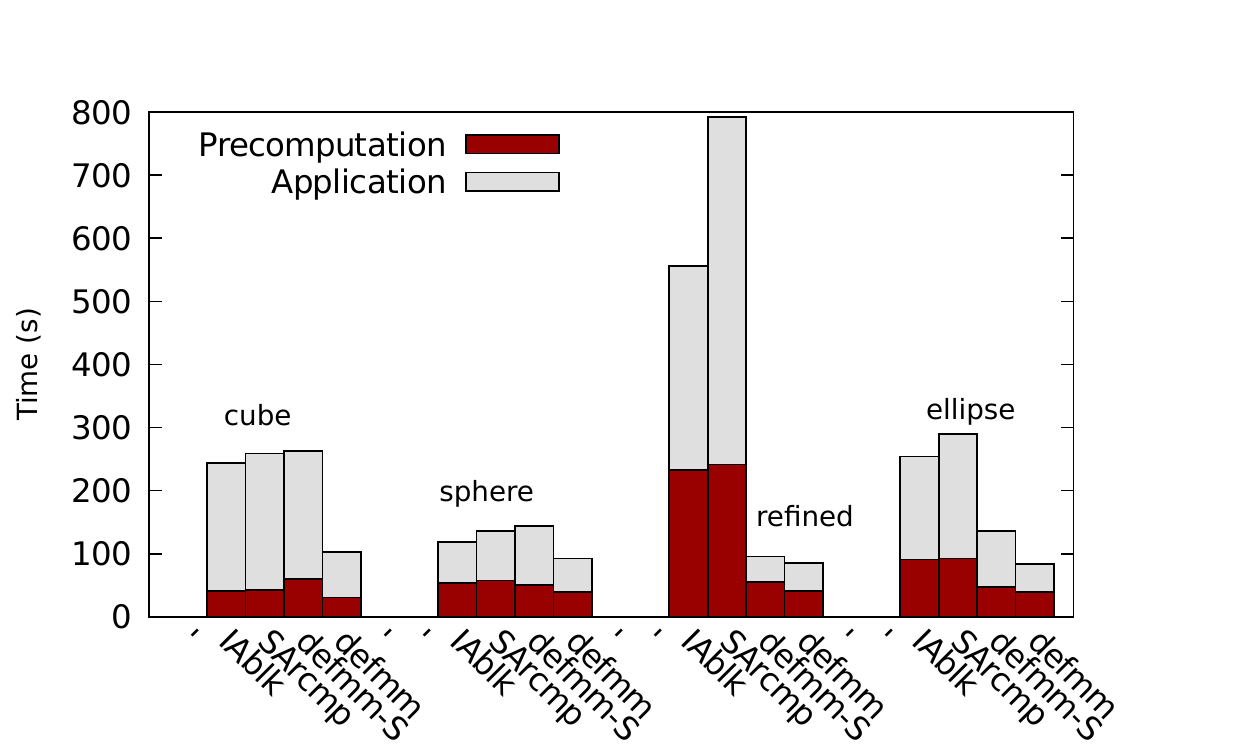}  
  \caption{$L=4$, $\kappa D = 0$}
  \label{fig:sub-first}
\end{subfigure}
\begin{subfigure}{.45\textwidth}
  \centering
  \includegraphics[trim={1cm 0 1cm 1cm},width=\linewidth]{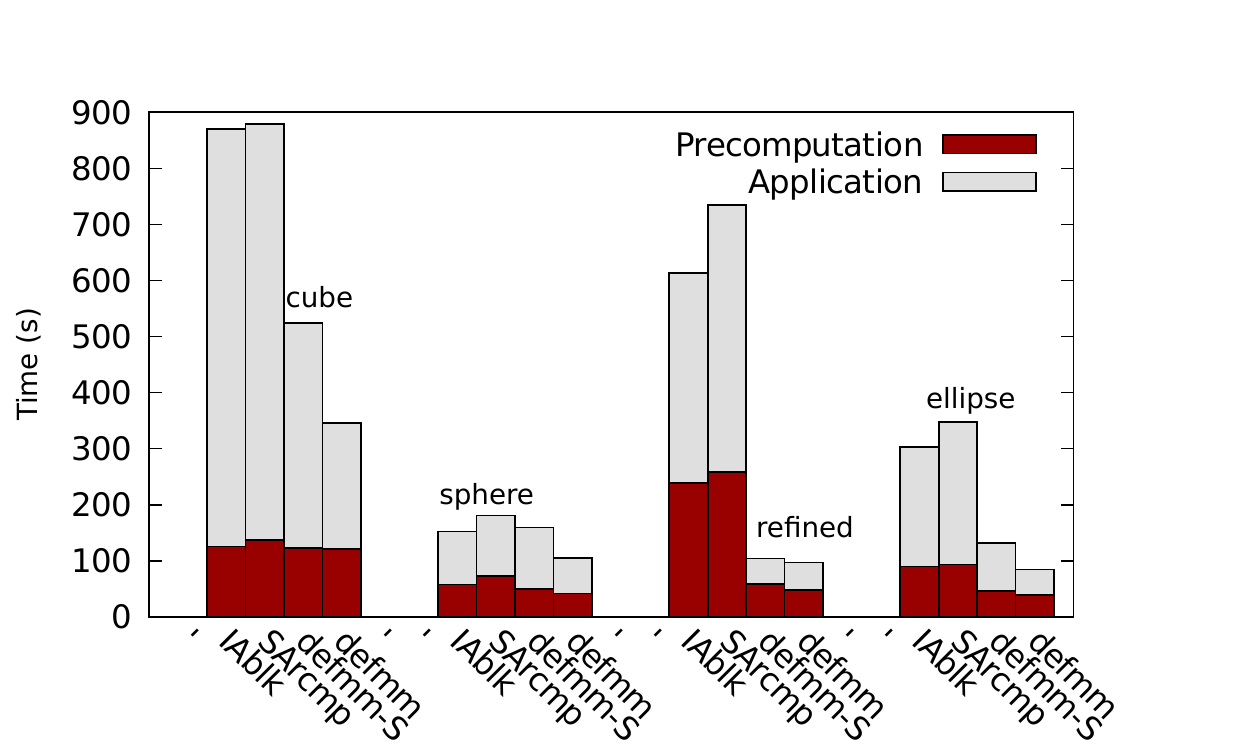}  
  \caption{$L=4$, $\kappa D = 64$}
  \label{fig:sub-second}
\end{subfigure}

\begin{subfigure}{.45\textwidth}
  \centering
  \includegraphics[trim={1cm 0 1cm 1cm},width=\linewidth]{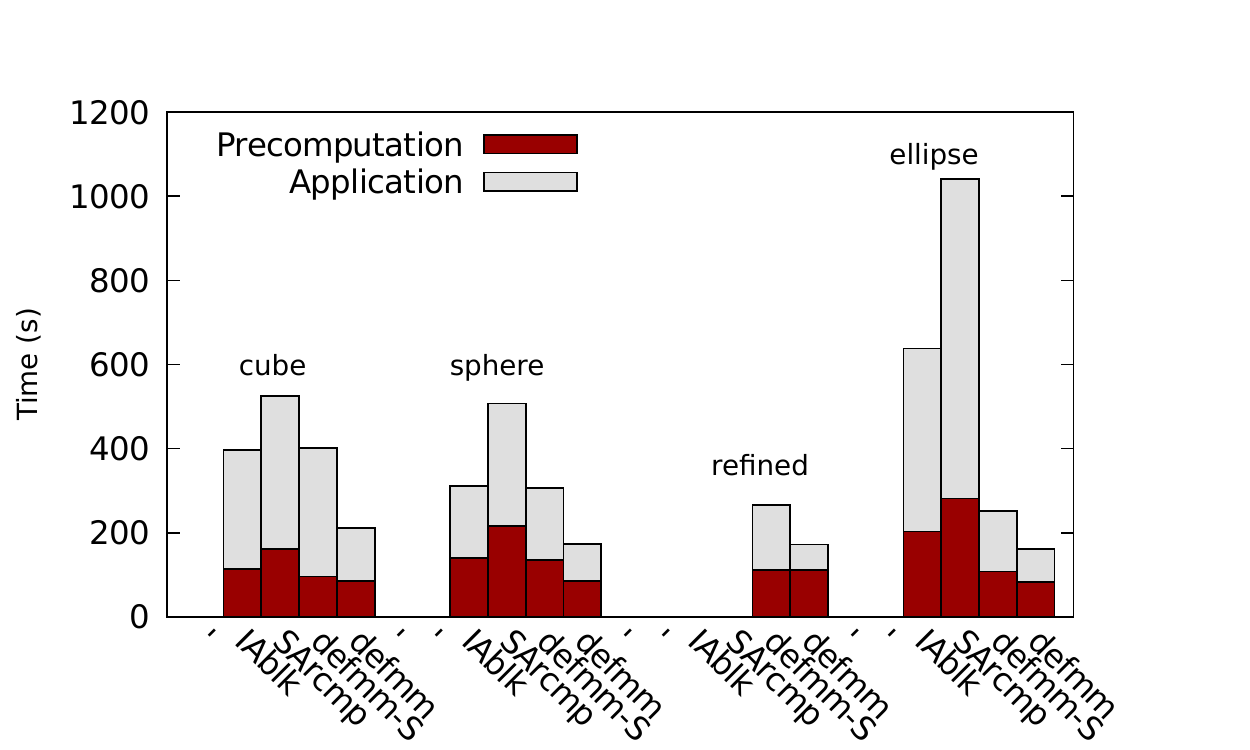}  
  \caption{$L=6$, $\kappa D = 0$}
  \label{fig:sub-third}
\end{subfigure}
\begin{subfigure}{.45\textwidth}
  \centering
  \includegraphics[trim={1cm 0 1cm 1cm},width=\linewidth]{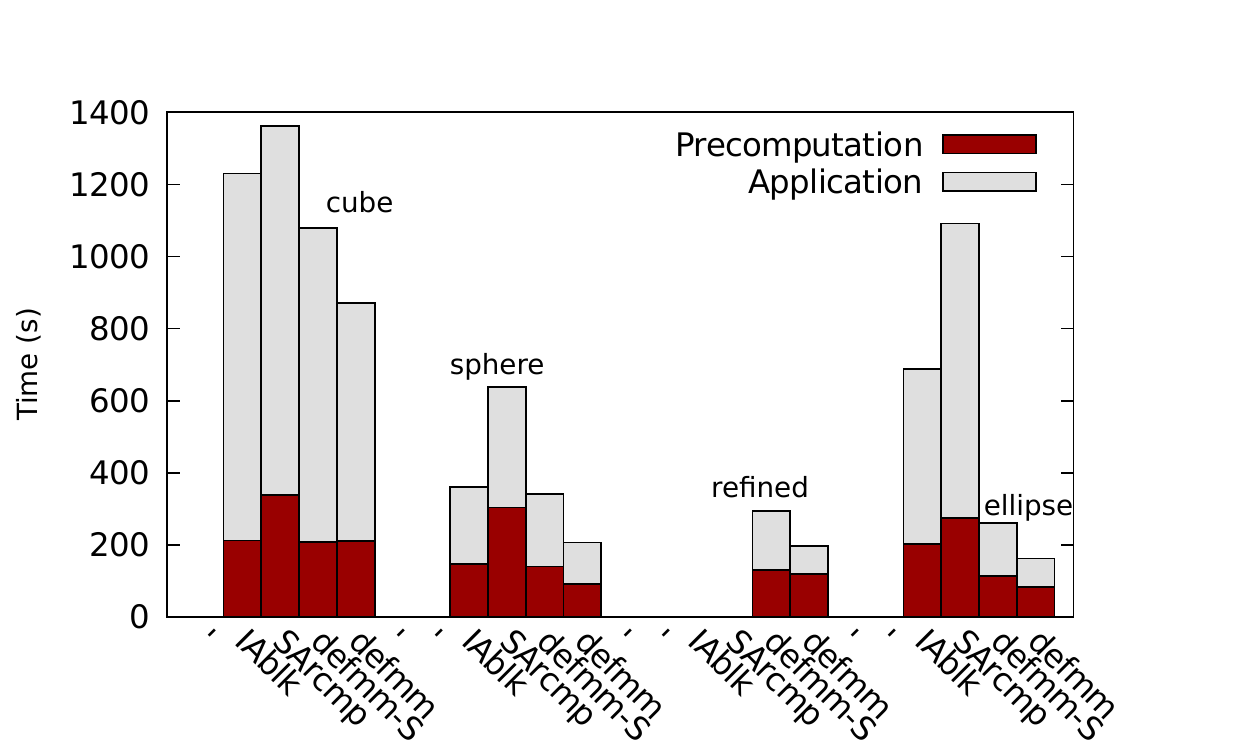}  
  \caption{$L=6$, $\kappa D = 64$}
  \label{fig:sub-fourth}
\end{subfigure}
\caption{Timings for one FMM execution 
 of the \textit{dfmm} and \textit{defmm} variants. The two tested variants of \textit{dfmm} are \textit{IAblk} 
 and \textit{SArcmp}. 
 \textit{defmm-S} corresponds to \textit{defmm} 
without 
P2P vectorization.}
\label{fig_compdefmmdfmm}
\end{figure}

\paragraph{Overall performance} 
Considering the overall times (including precomputation and application times), 
\textit{defmm} always outperforms \textit{dfmm}, being $1.3\times$ to $6.5\times$ faster than its best variant (\textit{IAblk} here). 
As detailed below, this is due to the different design of the two codes. 
Since the P2P operator is not vectorized in \textit{dfmm}, we also present in figure \ref{fig_compdefmmdfmm} a \textit{defmm} 
variant, denoted by \textit{defmm-S}, with a scalar (i.e. non-vectorized) P2P implementation. 
This enables us to more precisely study the impact of the other differences between \textit{defmm} and \textit{dfmm}.
\textit{defmm-S} offers similar or better performance than the best \textit{dfmm} variant on the uniform cube and on the sphere,
and outperforms \textit{dfmm} by factors $1.8\times$ to $5.9\times$ on more non-uniform distributions (i.e. the refined cube and the ellipse).
This shows 
that 
regarding directional FMMs our FFT-based approach  is 
competitive with or faster than the low-rank approximations used in \textit{dfmm}.

\paragraph{Sensitivity to the particle distribution} Since the distributions have the same number of particles ($\pm 1\%$), 
comparing two distributions (all other test parameters being identical) 
illustrates the sensitivity of the method to the particle distribution. Except for the uniform cube in the high-frequency regime 
(whose higher cost is due to its maximum number of effective directional expansions), the \textit{defmm} performance is few  
sensitive to the distribution, contrary to the \textit{dfmm} one. This is due to the differences in the tree construction and traversal  
(see section \ref{ss:design}): 
the combination of the $Ncrit$ strategy  
and of our specific DTT allows to better adapt in \textit{defmm} to the particle distribution
than the $MaxDepth$ strategy used in \textit{dfmm}.
Besides, the performance gap between \textit{IAblk} and \textit{defmm-S} is minimized for the (uniform) cube and the sphere (quasi-uniform on its surface) distributions.  These two distributions benefit indeed most from the BLAS routines and from the vector stacking in \textit{IAblk}.

\paragraph{Sensitivity to the wavenumber} 
Since both codes are based on the same directional approach, 
their performance  
is similarly impacted 
by an increase in the wavenumber.
Again, the uniform cube in the high-frequency regime is a special case (maximizing the number of effective expansions), 
where the \textit{dfmm} features (level-3 BLAS routines, cache management) lower here the performance gap with \textit{defmm} 
when the wavenumber increases: \textit{dfmm} is however still outperformed by \textit{defmm} in this case.  

\paragraph{Sensitivity to the interpolation order}

When moving from $L=4$ to $L=6$, the average performance ratio of \textit{defmm-S} over \textit{SArcmp} moves from $1.58\times$ to $2.40\times$ (for the uniform cube, the sphere and the ellipse test cases, and for any frequency).
This increasing performance gain with respect to the interpolation order is due to our FFT-based FMM. 
With respect to \textit{IAblk}, the performance gain of \textit{defmm-S} moves from $1.42\times$ to $1.56\times$: here the FFT gain is counterbalanced by the BLAS performance gain (which is greater for larger matrices). 

\paragraph{Precomputation costs} 
Thanks to our extension of the symmetry usage 
to the Fourier domain (see section \ref{ss:FFT_symmetries}), 
we manage to have on the uniform cube and on the sphere precomputation costs in \textit{defmm} as low as the \textit{dfmm} ones (considering the \textit{IAblk} variant which requires shorter precomputations than \textit{SArcmp}). 
In addition, in highly non-uniform distributions such as the refined cube and the ellipse, the precomputation costs of \textit{defmm} are drastically lower than the \textit{dfmm} ones, 
partly thanks to our $Ncrit$-based $2^d$-trees and to our blank passes (see section \ref{ss:design}) which minimize the required precomputations.

\section{Conclusion}
\label{section_conclusion}
In this paper we presented a new approach for the directional interpolation-based FMMs, using FFT techniques thanks to equispaced grids. We provided a consistency proof of the approximation process and showed how to extend the symmetries of the interpolation-based FMM to the Fourier domain. We detailed the algorithmic design of our \textit{defmm} library, as well as its high-performance optimizations on one CPU core. Finally, 
a comparison 
with a \textit{state-of-the-art} library exhibited the superior 
performance of our library 
in all test cases and in both high- and low-frequency regimes,
the \textit{defmm} performance being also few sensitive to the surface particle distribution. 

Future works will be dedicated to the \textit{defmm} parallelization. In this purpose, we will 
be able to lean on the dual tree traversal,
which is known to be highly efficient
regarding shared-memory parallelism 
(see e.g. \cite{yokota13,yokota18}),  
and on previous work (such as \cite{bensonpoulsontranengquistying14}) regarding distributed-memory parallelism. 
We also plan to integrate \textit{defmm} in an iterative solver in order to solve
complex realistic boundary integral equations.

\bibliographystyle{siamplain}
\bibliography{references}
\end{document}